%% file: main-revision.tex
\documentclass[11pt]{article}
\usepackage[letterpaper,top=2cm,bottom=2cm,left=3cm,right=3cm,marginparwidth=1.75cm]{geometry}
\usepackage[T1]{fontenc}
\usepackage[utf8]{inputenc}
\usepackage{hyperref}
\usepackage{float}
\usepackage[dvipsnames]{xcolor}
\include{startingpackages}

\newcommand{\bl}{\boldsymbol{\ell}}
\newcommand{\bt}{\boldsymbol{t}}
\newcommand{\bs}{\boldsymbol{s}}
\newcommand{\bu}{\boldsymbol{u}}
\newcommand{\bv}{\boldsymbol{v}}
\DeclareMathOperator{\PGL}{PGL}

\newcommand{\new}[1]{\textcolor{black}{#1}}

\newcommand{\minorideal}[1]{I_{#1}}
\newcommand{\linemultiviewvariety}[1]{\mathcal{L}_{#1}}
\newcommand{\IC}{\minorideal{\mathcal{C}}}
\newcommand{\ICarg}[2]{\minorideal{\mathcal{C}(#1, \ldots, #2)} }
\newcommand{\ICt}[2]{\ICarg{t_{#1}}{t_{#2}}}
\newcommand{\ICtm}{\minorideal{\mathcal{C}(\bt)} }

\newcommand{\ICtbarm}{\minorideal{\mathcal{C}(\bs)} }
\newcommand{\GCt}[2]{G_{\mathcal{C}(t_{#1}, \ldots ,t_{#2})}}
\newcommand{\LC}{\linemultiviewvariety{\mathcal{C}}}

\newcommand{\collICT}{\widetilde{I}_{\mathcal{C}(\bv)}}
\newcommand{\speccollICT}{\widetilde{I}_{\mathcal{C}(\bu)}}

\DeclareMathOperator{\GL}{GL}

\usepackage{listings}
\definecolor{codedarkgreen}{RGB}{51, 133, 4}
\definecolor{codemaroon}{RGB}{133, 5, 63}
\definecolor{codeteal}{RGB}{0, 128, 96}
\lstdefinelanguage{Macaulay2}{
basicstyle=\small\ttfamily,
alsoletter=",
classoffset=1,
keywords={sub,join,matrix,minors,gb,transpose,det,ideal,apply,subsets,ker,gens,fold,flatten,entries},
keywordstyle={\color{blue}},
classoffset=2,
morekeywords={from, to, list},
keywordstyle={\color{codemaroon}},
classoffset=3,
morekeywords={QQ},
keywordstyle={\color{codedarkgreen}},
classoffset=4,
morekeywords={MonomialOrder},
keywordstyle={\color{codeteal}},
xleftmargin=1.5cm,
xrightmargin=1em,
columns=fullflexible,
keepspaces=true,
stepnumber=1,
numbers=none,
captionpos=b,
showspaces=false,
frame=none
}

\title{Line Multiview Ideals}

\date{}
\author[1]{Paul Breiding}
\author[2]{Timothy Duff}
\author[3]{Lukas Gustafsson}
\author[3]{Felix Rydell}
\author[1,4]{Elima Shehu}
\affil[1]{{\small Universit\"at Osnabr\"uck, Osnabr\"uck, Germany}}
\affil[2]{{\small University of Washington, Seattle, USA}}
\affil[3]{{\small KTH Royal Institute of Technology, Stockholm, Sweden}}
\affil[4]{{\small Max Planck Institute for Mathematics in the Sciences, Leipzig, Germany}}

\begin{document}

\maketitle
\abstract{We study the following problem in computer vision from the perspective of algebraic geometry: Using $m$ pinhole cameras we take $m$ pictures of a line in $\mathbb P^3$. This produces $m$ lines in $\mathbb P^2$ and the question is which $m$-tuples of lines can arise that way. We are interested in polynomial equations and therefore study the complex Zariski closure of all such tuples of lines. The resulting algebraic variety is a subvariety of $(\mathbb P^2)^m$ and is called line multiview variety. In this article, we study its ideal. We show that for generic cameras the ideal is generated by $3\times 3$-minors of a specific matrix. We also compute Gr\"obner bases and discuss to what extent our results carry over to the non-generic case.}

\medskip
\section{Introduction and Main Results}\label{sec:background}
A \new{\emph{pinhole camera}} is a projective linear map
$$
\mathbb P^3 \dashrightarrow \mathbb P^2,\quad x\mapsto Cx,$$
where $C\in\mathbb C^{3\times 4}$ \new{is} rank 3. The symbol $\dashrightarrow$ indicates that the map is not defined everywhere; it is not defined at the \emph{camera center} $c=\ker C$. 

Suppose now that $\mathcal C=(C_1,\ldots, C_m)\in(\mathbb C^{3\times 4})^m$ is an arrangement of $m$ cameras. Throughout this article, we assume that the centers are distinct and that there are at least two cameras.
We are interested in the tuples of lines that arise by taking pictures of a common line $L\subset\mathbb P^3$ using the $m$ cameras in $\mathcal C$. To make this precise, let us denote by $\mathbb{G}$ the Grassmannian of lines in $\PP^3$. 
\new{
A line in the image plane $\PP^2$ is represented by a linear equation~$x^{T}\ell=0$ for some $\ell \in \PP^2.$ 
For this reason, it is convenient to also denote the \emph{dual} of the image plane by $\PP^2$.}
We consider the \textit{joint camera map}
\begin{equation}\label{def_camera_map}
    \Upsilon_{\mathcal C}: \mathbb G \dashrightarrow (\mathbb P^2)^m,\quad
    L\mapsto (\ell_1,\ldots,\ell_m),
\end{equation}
where $\ell_i$ is the linear equation for the $i$-th image line $C_i\cdot L$.
\new{More explicitly,} if $L$ is spanned by two points $a,b\in \PP^3$  then $\ell_i=(C_ia)\times (C_ib)$, where $\times$ is the cross product. 
The Zariski closure of the image of this map is the \emph{line multiview variety} of $\mathcal{C},$
$$\LC := \overline{\Upsilon_{\mathcal C}(\mathbb G)}.$$
The \emph{ideal} $I(\LC)$ of $\LC$ \new{is the ideal of all polynomials that vanish on $\LC$, or equivalently of all polynomials that vanish on $\Upsilon_{\mathcal C}(\mathbb G)$}. The goal of this paper is to study the ideal $I(\LC)$\new{, solving} the implicitization problem~\cite[\S 3.3]{CLO15} for line multiview varieties.

Multiview varieties are fundamental objects in algebraic vision, a field of research that applies the tools of algebraic geometry and neighboring fields to topics in computer vision.
So far, most attention has been paid to multiview varieties that model 3D scenes involving points.
These point multiview varieties and the equations defining them were originally studied in the computer vision literature under various guises (eg.~the \emph{natural descriptor}~\cite{heyden1997algebraic}, the \emph{joint image}~\cite{Tri95,DBLP:conf/iccv/TragerHP15}, or via the Grassmann-Cayley algebra in~\cite{DBLP:conf/iccv/FaugerasM95}.)
From the point of view of algebraic vision, the results of~\cite{AST13,agarwal2019ideals} characterize the \emph{vanishing ideal} of the point multiview variety of a suitably generic camera arrangement.
In particular, these results solve the implicitization problem for point multiview varieties. 
While the point multiview variety has received much attention, the line multiview variety has less so. The study of line multiview varieties was initiated in \cite{line-multiview}, which focused on the geometric structure of $\LC$. Our paper extends this study, focusing on the \emph{line multiview ideal} $I(\LC)$.

In \cite[Theorem 2.1]{line-multiview} it is shown that $\LC$ is irreducible and $\dim \LC=\new{\dim \mathbb G}=4$. 
The line multiview ideal belongs to the polynomial ring
\begin{equation}\label{def_R}
R:=\mathbb C[\ell_{i,j} \mid 1\leq i\leq 3,1\leq j\leq m].\end{equation}
We write $\ell_{k}=(\ell_{i,k})_{i=1}^3$, and denote by $M(\ell)$ the matrix
\begin{equation}\label{eq:Ml}
    M(\ell) = \begin{bmatrix}
    C_1^T \ell_1 & \cdots & C_m^T \ell_m
    \end{bmatrix} \in R^{4\times m}.
\end{equation}
By \cite[Theorem 2.5]{line-multiview}, we have
$\LC= \left\{\ell \in (\mathbb P^2)^m \mid \mathrm{rank}\ M(\ell) \leq 2\right\}$
if and only if no four camera centers of the $C_i$ are collinear. We refer to this as the \emph{generic case}. If four camera centers are collinear we have $\LC\subsetneq \left\{\ell \in (\mathbb P^2)^m \mid \mathrm{rank}\ M(\ell) \leq 2\right\}$ and the remaining equations to describe $\LC$ are discussed in \cite[Theorem 2.6]{line-multiview} and treated with greater detail in \Cref{s: Set-Theory} of this article. The geometric idea behind the definition of~$M(\ell)$ is that $C_i^T \ell_i$ defines the plane in $\mathbb P^3$ projecting to $\ell_i$ under $C_i$ (called \emph{back-projected plane}), and there is a line in the intersection of all back-projected planes if and only if the rank of $M(\ell)$ is at most 2. \new{The geometry is illustrated in~\Cref{fig:three-view}.}

\begin{figure}
    \centering
\begin{tikzpicture}[scale=0.35]
\draw[fill] (1,5) circle (4pt) node[below] {${c_1}$};
\draw[fill] (11.5,1) circle (4pt) node[below] {${c_2}$};
\draw[fill] (22,5) circle (4pt) node[below] {${c_3}$};
\draw (1,8) -- (1,14);
\draw (1,8) -- (6,5);
\draw (6,5) -- (6,11);
\draw (1,14) -- (6,11)  node[above=8] {${H_1}$};
\draw (9,3) -- (9,9);
\draw (9,3) -- (14,3);
\draw (14,3) -- (14,9);
\draw (14,9) -- (9,9) node[above=8, midway] {${H_2}$};
\draw (17,4) -- (17,10);
\draw (17,4) -- (22,7);
\draw (22,7) -- (22,13);
\draw (17,10) -- (22,13) node[left=35] {${H_3}$};

\fill[ teal!10] (1,8) -- (6,5) -- (6,11) -- (1,14) -- cycle;
\fill[ teal!10] (9,3) -- (14,3) -- (14,9) -- (9,9) -- cycle;
\fill[ teal!10] (17,4) -- (22,7) -- (22,13) -- (17,10) -- cycle;

\draw[line width=1.8pt, violet] (2.6,8.5) -- (4.4,8) node[black, above, midway, font=\bfseries] {${\ell_1}$};
\draw[line width=1.8pt,  violet] (10.5,4.4) -- (12.7,5) node[black, above, midway, black!=\bfseries] {${\ell_2}$};
\draw[line width=1.8pt,  violet] (18.5,8) -- (20.5,9) node[black, above, midway, font=\bfseries] {${\ell_3}$};
\draw (1,5) -- (2.6,8.5);
\draw[dotted] (2.6,8.5) -- (4.2, 12);
\draw (4.2,12) -- (7,18);
\draw (1,5) -- (4.4,8);
\draw[dotted] (4.4,8) -- (5.99375, 9.40625);
\draw (5.99375, 9.40625) -- (8.4375, 11.5625);
\draw[dotted] (8.4375, 11.5625) -- (17,19);
\draw (11.5,1) -- (10.5,4.4);
\draw[dotted] (10.5,4.4) -- (9,9);
\draw (9,9) -- (7,18);
\draw (11.5,1) -- (12.7,5);
\draw[dotted] (12.7,5) -- (14,9);
\draw(17, 19) -- (14,9);
\draw (17.0063, 9.1375) -- (14.6,11.2);
\draw[dotted] (17.0063, 9.1375) -- (18.5,8);
\draw[dotted] (7,18) -- (14.6,11.2);
\draw (18.5,8) -- (22,5);
\draw (20.5,9) -- (22,5);
\draw[dotted] (20.5,9) -- (19.5625, 11.5);
\draw ((19.5625, 11.5) -- (17,19);

\draw[line width=1.5pt, violet] (7,18) -- (17,19) node[black, above] {${L}$};
\end{tikzpicture}
    \caption{Illustration of projecting three images of a line in three-dimensional space. The purple line, denoted by $L$, represents the actual line in 3D space. The three small purple lines represent three images of the line $L$, captured from different perspectives $c_i$. Additionally, transparent planes are shown, representing the backprojected planes $H_i$ defined via $C^T_i\ell_i$ corresponding to the captured images. This visualization demonstrates the relationship between the original 3D line, its multiple images, and the corresponding backprojected planes.}
    \label{fig:three-view}
\end{figure}

To any camera arrangement, we may also associate the ideal
$$\IC := \left\langle 3\times3\text{-minors of } M(\ell)\right\rangle.$$
Our first main result of this paper is the following.
\begin{theorem}\label{thm:main_ideal} 
Let $\mathcal{C}$ be an arrangement of $m$ cameras such that no four cameras are collinear. Then, the vanishing ideal of $\LC$ is generated by the $3\times3$-minors of $M(\ell)$; i.e.,
\[I(\LC) = \IC.\]
\end{theorem}
We prove this theorem at the end of \Cref{sec:generators}.

Next, we study \emph{Gr\"obner bases} for the ideal $\IC$. Recall that a Gröbner basis is a specific type of generating set for an ideal. \new{A Gr\"obner basis} exist for every \emph{monomial order}; see, e.g., \cite[\S 2.5 Corollary 6]{CLO15}. A monomial order defines the notion of \emph{leading term} of a polynomial. The definition of a Gröbner basis $G$ is that the leading term of any polynomial in $I$ is divisible by one of the leading terms of an element in $G$. 

There are $N:=4\binom{m}{3}$ $3\times3$-minors of the $4\times m$-matrix $M(\ell)$.  Let us denote them by~$m_1,\ldots,m_N$, so that, by \Cref{thm:main_ideal},
$
    \IC = \langle m_1,\dots,m_{N}\rangle.
$
The first question is to decide whether $B:=\{m_1,\ldots,m_N\}$ is a Gröbner basis for \emph{some} monomial order or not. For this, we implemented the \emph{Gröbner basis detection \new{a}lgorithm} in \texttt{Macaulay2} \cite{M2} \new{as a part of the package \texttt{SagbiGbDetection}~\cite{SagbiGbDetectionSource}}. This algorithm, first described in~\cite{GriStu93} (see also~\cite[Ch.\ 3]{Stu96}), consists of two main steps. The first step involves polyhedral computations: we compute the Newton polytope $\mathrm{Newt}(B):= \mathrm{Newt}(m_1\cdots m_N)\subseteq \mathbb{R}^{3m}$, collect all vertices of $\mathrm{Newt}(B)$ whose normal cones intersect the positive orthant, and determine a \emph{weight order} from each of these cones. The second step then uses Buchberger's criterion to check if $B$ is a Gröbner basis with respect to each weight order. In our case, already for $m = 3$ the minors do not form a Gröbner basis for \emph{any} monomial order.

This opens the follow-up question: which polynomials in the ideal $\IC$ \emph{do} form a Gr\"obner basis? We discuss this in \Cref{sec:groebner}. In our study we restrict ourselves to specific monomial orders. Computing Gr\"obner bases for $\IC$ for \new{every} ordering remains an open problem. We emphasize that for the point multiview variety, Aholt, Sturmfels, and Thomas \new{give} a \emph{universal Gr\"obner basis}; that is, a subset of polynomials that is a Gr\"obner basis for \new{every} monomial order; see \cite[Theorem 2.1]{AST13}. For $m\ge 4,$ the structure of this universal Gr\"{o}bner basis is completely determined by its restrictions to subsets of four cameras, whose elements correspond to the $2,3,$ and $4$-view tensors of multiview geometry \cite[Ch.~17]{HZ04}.
Intriguingly, in our case, we obtain Gr\"obner bases that are determined by its restrictions to subsets of cameras of size \emph{five} rather than four. For $m=2,3,4$ one can compute  Gr\"obner basis of $\IC$ using \texttt{Macaulay2}. For $m\geq 5$, we have the following theorem.
\begin{theorem}\label{thm:main2}
Suppose that $m\geq 5$ and all camera matrices are of the form
$$C_i = \begin{bmatrix}
1 & 0 & 0 & s_{1,i } \\
0 & 1 & 0 & s_{2,i } \\
0 & 0 & 1 & s_{3,i } 
\end{bmatrix}$$
and that $\mathbf s=(s_{j,i})_{1\leq j\leq 3, 1\leq i\leq m}$ is generic. Then, the Gr\"obner basis $G_m$ for the GRevLex order consists of polynomials that are supported on at most five cameras. More specifically, 
$$G_m = \bigcup_{\sigma\in \binom{[m]}{5}} G_\sigma,$$ 
where $G_\sigma$ is the Gröbner basis with respect to the GrevLex order of the line multiview ideal involving only the cameras with indices in $\sigma$.
\end{theorem}
\begin{proof}
This follows from \Cref{thm:GB} and \Cref{thm:gb-LMI} below. In fact, \Cref{thm:gb-LMI} defines explicitly what it means for $\mathbf s$ to be generic.
\end{proof}
The assumption on the shape of camera matrices in \Cref{thm:main2} may appear to be restrictive. However, as we observe in~\Cref{sec:group} that the group $\mathrm{PGL}_4 \times \mathrm{PGL}_3^m $ which acts on camera arrangements by
\begin{align}\label{def_group_action}
h\cdot (C_1, \ldots, C_m) = (H_1 C_1 H^{-1}, \ldots, H_m C_m H^{-1}), \text{ where } h=(H, H_1, \ldots, H_m),
\end{align}
also acts on line multiview ideals as 
$
\mathcal{L}_{h\cdot \mathcal{C}} =L_h(\LC) 
$
or, equivalently,
$
 I(\mathcal{L}_{h\cdot \mathcal{C}})  = L_h^{-1} (I(\LC))$, where $L_h$ is defined by $L_h(\ell_i)= H_i^{-T} \ell_i$; see \Cref{prop: group action}. Any camera arrangement $\mathcal C$ can be transformed into an arrangement of the form specified in \Cref{thm:main2} by a suitable choice of $h\in \mathrm{PGL}_4 \times \mathrm{PGL}_3^m$. Moreover, one can find $h$ such that $h\cdot \mathcal C$ is general in the sense of \Cref{thm:main2}, if and only if no three camera centers in $\mathcal C$ are collinear. We prove this fact in \Cref{prop:when_is_transport_possible}.
 The group action can then be exploited to solve computational problems for line multiview ideals.
 For example, consider the problem of ideal membership. Suppose that we have $f\in R$ and want to decide whether $f\in I(\LC)$. Since $f\in I(\LC)$, if and only if $L_h^{-1}(f)\in L_h^{-1}(I(\LC))= I(\mathcal{L}_{h\cdot \mathcal{C}})$, we can use a Gröbner basis of the latter for the division algorithm from \cite[\S 2.6]{CLO15}. 

Finally, in the last part of the paper, Sections \ref{s: Set-Theory} and \ref{s: GrobColl}, we prove variants of \Cref{thm:main_ideal} and \Cref{thm:main2} for the case when \emph{all}  cameras are collinear. More specifically, in \Cref{s: Set-Theory}, we give an explicit set-theoretic description of line multiview varieties with arbitrary camera arrangements. This is an improvement over the treatment in \cite{line-multiview}, where the polynomial equations are described via elimination. This includes cases when there are 4 collinear cameras or when all cameras are collinear. \new{Note that we use the term \emph{collinear cameras} to refer to cameras that have collinear centers.} Next, in \Cref{s: GrobColl} we adapt the results from \Cref{s: Set-Theory} to the ideal-theoretic methods of \Cref{sec:groebner}. We produce a Gröbner basis for the multiview ideal for an arrangement of $m \geq 4$ collinear cameras. Notably, this Gröbner basis is, analogously to the generic case discussed above, determined by its restrictions to subsets of four cameras.

\subsection*{Acknowledgements} The research of Elima Shehu and Paul Breiding was funded by the Deutsche Forschungsgemeinschaft (DFG, German Research Foundation), Projektnummer 445466444. Tim Duff acknowledges support from an NSF Mathematical Sciences Postdoctoral Research Fellowship (DMS-2103310). Felix Rydell was supported by the Knut and Alice Wallenberg Foundation within their WASP (Wallenberg AI, Autonomous Systems and Software Program) AI/Math initiative. Lukas Gustafsson was supported by the VR grant [NT:2018-03688]. 

Part of this work was completed while the authors visited the Czech Institute of Informatics, Robotics, and Cybernetics as part of the Intelligent Machine Perception Project. We thank Tomas Pajdla for the invitation.

\medskip
\section{The Line Multiview Ideal For Generic Cameras}\label{sec:generators}
The goal of this section is to prove \Cref{thm:main_ideal}. For this, we introduce the \emph{cone} over the line multiview variety
\begin{equation}\label{def_cone_L_C}
\widehat{\mathcal L_{\mathcal C}} := \{\ell \in (\mathbb C^3)^m \mid \operatorname{rank} M(\ell) \leq 2\}.
\end{equation}
The key step for proving \Cref{thm:main_ideal} is to prove the following result. Recall that we assume all centers of an arrangement are distinct and that $m\geq 2$.
\begin{proposition}\label{thm:almost_main_ideal}
If no four cameras are collinear, then $\IC=I(\widehat{\mathcal L_{\mathcal C}})$.
\end{proposition}

The idea for proving this proposition is to show that $\IC$ is a Cohen-Macaulay ideal in the case when no four cameras are collinear. We do so in \Cref{prop_CM} below and use this result to deduce in \Cref{reduced_ring} that~$R/\IC$ is reduced. We formally give the proof of  \Cref{thm:almost_main_ideal} together with a proof of \Cref{thm:main_ideal} at the end of this section.

We first need two lemmata.
\begin{lemma}\label{le: XC} Let $\mathcal{C}$ be an arrangement such that no four \new{cameras} are collinear. 
\begin{enumerate}
\item  Denote
$X_{\mathcal{C}}:=\{\ell =(\ell_1,\ldots,\ell_m)\in \widehat{\mathcal L_{\mathcal C}}\mid \ell_i\neq 0 \text{ for } 1\leq i\leq m\}.$
Then $X_{\mathcal{C}}$ is a Zariski dense subset of $\widehat{\mathcal L_{\mathcal C}}$, meaning $\overline{X_{\mathcal{C}}}=\widehat{\mathcal L_{\mathcal C}}$. 
\item $\widehat{\mathcal L_{\mathcal C}}\subset (\CC^3)^m$ is the closure of the image of the following map,
\begin{align*}
\widehat{\Upsilon}_\mathcal{C}:\CC^{4} \times \CC^{4} \times \CC^m &\dashrightarrow (\CC^3)^m,\quad (x, y, \lambda_1, \ldots , \lambda_m) \mapsto (\ell_1,\ldots,\ell_m), 
\end{align*}
where $\ell_i =  \lambda_i\, (C_ix)\times (C_iy)$ and $\times$ denotes the cross-product in $\mathbb C^3$ (so if $\ell_i\neq 0$, it is the equation of the projective line passing through $C_ix$ and $C_iy$).
\end{enumerate}
\end{lemma}

\begin{proof} \new{We show that $\widehat{\mathcal L_{\mathcal C}}$ lies in the Euclidean closure of $X_{\mathcal{C}}$. Let $\ell \in \widehat{\mathcal L_{\mathcal C}}\setminus X_{\mathcal{C}}$ and let $J\subseteq[m]$ denote the set of indices for which $\ell_i\neq 0$. Observe that if $\ell_i=0$, then the generators of $I_\mathcal{C}$ that involve the variables of $\ell_i$ are zero, because they are homogeneous in $\ell_i$. The remaining generators define the ideal $I_{\mathcal{C}'}$, where $\mathcal{C}'$ is the arrangement we get by removing $C_i$ from $\mathcal C$. In particular, let $$\pi_J: \widehat{\mathcal L_{\mathcal C}}\to \widehat{\mathcal L_{\mathcal C_J}}$$
be the coordinate projection, where $\mathcal{C}_J$ denotes the arrangement of cameras corresponding to the indices of $J$. Then $\pi_J(\ell)\in (\CC^3)^{J}$ is a representative of an element of $\mathcal{L}_{\mathcal{C}_J}$.}

\new{In order to show $\ell\in \overline{X_{\mathcal{C}}}$, it suffices to find a point $\ell'\in \LC$ (which we identify with a representative in $(\CC^3)^m$) such that $\pi_J(\ell)=\pi_J(\ell')$. This is because we can create a sequence ${\ell}_\epsilon\in X_\mathcal{C}$ converging to $\ell$ as $\epsilon\to 0$  by letting $({\ell}_{\epsilon})_i=\ell_i\in\CC^3$ whenever $i\in J$ and  $({\ell}_{\epsilon})_i=\epsilon\ell_i'\in\CC^3$ otherwise.}

\new{We can find such a $\ell'$ trivially if $|J|=0$. If $|J|=1$, say $J=\{i\}$, then let $\ell'$ be the image of any line $L$ in $H_i$ meeting no center under the joint camera map $\Upsilon_\mathcal{C}$. If $|J|\ge 2$, since~$\pi_J(\ell)$ represents an element of $\mathcal{L}_{\mathcal{C}_J}$, \cite[Proposition 2.4 2.]{line-multiview} says that there is an element $\ell'\in \mathcal{L}_{\mathcal C}$ that projects onto $\pi_J(\ell)$ via $\pi_J$.}


For the second part, note that the set $Y_\mathcal{C}\subseteq\CC^4\times \CC^4\times \CC^m$ of points $(x,y,\lambda_1,\ldots,\lambda_m)$ such that the line spanned by $x,y$ in $\PP^3$ contains no center and $\lambda_i\neq 0$ is an open dense subset. Its image under $\widehat{\Upsilon}_\mathcal{C}$ is a subset of $X_\mathcal{C}$ and its closure is a subset of $\widehat{\mathcal L_{\mathcal C}}$. For the other direction, take $\ell\in \widehat{\mathcal L_{\mathcal C}}$. By the first statement, let $\ell^{(n)}=(\ell^{(n)}_1,\ldots,\ell^{(n)}_m)\to \ell$ with $\ell^{(n)}\in X_\mathcal{C}$. The projective class of $\ell^{(n)}$ in $(\PP^2)^m$ lies in $\LC$. Fix an $n$. Since $\LC$ is the Euclidean closure of the image $\Upsilon_\mathcal{C}(\mathbb G)$ (by Chevalley's theorem \cite[Theorem 4.19]{michalek2021invitation}), there is a sequence of lines $L^{(k)}$ meeting no centers such that for some nonzero scaling $\lambda_i^{(k)}$ we have $\ell^{(n)}_i=\lim_{k\to\infty } \lambda_i^{(k)}C_i\cdot L^{(k)}$.
Therefore the images of $(x^{(k)},y^{(k)},\lambda_1^{(k)},\ldots,\lambda_m^{(k)})$ under the map~$\widehat{\Upsilon}_\mathcal{C}$ converge to $\ell^{(n)}$ so that $\ell^{(n)}$ is in the closure of the image for each $n$. This implies that the limit $\ell$ of $\ell^{(n)}$ is also in the closure of the image of $\widehat{\Upsilon}_\mathcal{C}$. 
\end{proof}

\begin{lemma}\label{lem_dimension_cone} $\widehat{\mathcal L_{\mathcal C}}$ is irreducible and of dimension $4+m$. 
\end{lemma}
\begin{proof} By the second statement of \Cref{le: XC}, $\widehat{\mathcal L_{\mathcal C}}$ is the closure of the image of an irreducible variety under a rational map. This means that it is irreducible. Consider $X_\mathcal{C}$ as in \Cref{le: XC}. The projection $\pi: X_\mathcal{C}\to \mathcal L_{\mathcal C}$ is surjective and has $m$-dimensional fibers. So $\dim \overline{X_\mathcal{C}} = \dim  \mathcal L_{\mathcal C} + m$. Moreover, $\overline{X_\mathcal{C}} =\widehat{\mathcal L_{\mathcal C}}$ by the first statement of \Cref{le: XC} and $\dim \mathcal L_{\mathcal C}= 4$ by \cite[Theorem 2.1]{line-multiview}.
\end{proof}

The goal is now to prove that $\IC$ is a \emph{Cohen-Macaulay ideal}. Let us recall the definition of this: A unital, commutative, and Noetherian ring $S$ is called a \emph{Cohen-Macaulay ring}, if $\dim S = \operatorname{depth} S$; see, e.g., \cite[Ch.\ 5]{Decker_Lossen}. An ideal $I$ is a Cohen-Macaulay ideal if $S=R/I$ is a Cohen-Macaulay Ring. 

To prove that $\IC$ is a Cohen-Macaulay ideal we need the concept of \emph{codimension} for ideals in $R=\mathbb C[\ell_{i,j} \mid 1\leq i\leq 3,1\leq j\leq m].$
Let first $J\subset R$ be a prime ideal. The {codimension} of $J$ is defined to be
$\operatorname{codim} J := \dim R_J,$
where $R_J$ is the localization of~$R$ at~$J$ and $\dim R_J$ is the Krull dimension. Equivalently, $\operatorname{codim} J$ is the maximal length~$k$ of a chain of prime ideals of the form $P_0\subsetneq \cdots\subsetneq P_k=J$. This equivalence follows from the bijection of ideals of $R$ contained in $J$ and ideals of $R_J$ by the map $r\mapsto \tfrac{r}{1}$. For any ideal $I\subset R$ its codimension is then defined as
$$\operatorname{codim} I := \min_{I\subset J,\; J \text{ prime}} \operatorname{codim} J.$$ 
It follows from the definitions that for all ideals $I\subset R$ we have
$
    \dim  R/I +\operatorname{codim} I \leq \dim R.
$ 
By \cite[Lemma 11.6 (b)]{Gathmann:CommAlgebra} we have for a prime ideal $J\subset R$
\begin{equation}\label{eq1}\dim R/J+\operatorname{codim} J=\dim R.\end{equation}
Moreover, since $R$ is a \emph{polynomial} ring,~\cite[Corollary 13.4]{Eisenbud} implies that~\eqref{eq1} holds for any ideal $J \subset R.$
We use these facts to prove the following result.
\begin{proposition}\label{prop_CM}
If no four cameras are collinear, $\IC$ is a Cohen-Macaulay ideal.
\end{proposition}
\begin{proof}
Recall that $M(\ell) \in R^{4\times m}$. We now specialize \cite[Theorem 2.25]{Decker_Lossen} to $k=3$ and $p=4,q=m$ (see also \cite[Section 18]{Eisenbud}). This result shows that $\IC$ is a Cohen-Macaulay ideal if $\operatorname{codim} \IC = (p-k+1)(q-k+1)=2m-4$. We show the latter. 

The zero set of $I_{\mathcal C}$ in $(\mathbb C^3)^m$ is $\widehat{\mathcal L_{\mathcal C}}$. This implies $I(\widehat{\mathcal L_{\mathcal C}})=\sqrt{I_{\mathcal C}}$. Recall that $\sqrt{I_{\mathcal C}}$ is the intersection of all prime ideals containing $I_{\mathcal{C}}$. 
From this and the definition of codimension, it follows that $\operatorname{codim} I_{\mathcal C}=\operatorname{codim}I(\widehat{\mathcal L_{\mathcal C}})$. Thus~\Cref{lem_dimension_cone} implies
$\dim R/I(\widehat{\mathcal L_{\mathcal C}}) = 4+m$. 
Using~\Cref{eq1}, we conclude
 \phantom\qedhere
\[\pushQED{\qed} \operatorname{codim} I(\widehat{\mathcal L_{\mathcal C}}) = \dim R-(4+m) =2m-4.\qedhere
\popQED\]
\end{proof}
The next result we need for proving \Cref{thm:main_ideal} is that the quotient ring for the determinantal ideal is reduced. The proof relies on \Cref{prop_CM}.
\begin{proposition}\label{reduced_ring}
If no four cameras are collinear, then $R/\IC$ is reduced. 
\end{proposition}

\begin{proof}
Denote by $m_1,\ldots, m_N$, where $N=4\binom{m}{3}$, the $3\times 3$ minors of $M(\ell)$. Then, we have $\IC=\langle m_1,\dots, m_N\rangle$. We have shown in the proof of \Cref{prop_CM} that 
$$c:=\operatorname{codim}(\IC)=2m-4.$$
Consider the Jacobian matrix 
$$\mathrm{Jac} := \begin{bmatrix}\
\frac{\partial m_k}{\partial \ell_{i,j}}\ 
\end{bmatrix}_{k\in\{1,\ldots, N\}, (i,j)\in\{1,2,3\}\times\{1,\ldots,m\}}\in\mathbb C^{N\times (3m)}.$$
We denote by $f_1,\ldots,f_M$, where $M:=\tbinom{N}{c}\cdot \tbinom{3m}{c}$, the $c\times c$ minors of $\mathrm{Jac}$. 
Let~us~consider the ideal generated by these minors modulo $I_{\mathcal C}$; i.e., we consider $J:=\langle f_1,\ldots,f_M\rangle / \IC \subset R/ \IC$. Since $R/I_{\mathcal C}$ is Cohen-Macaulay by \Cref{prop_CM}, we know that $R/I_{\mathcal C}$ is reduced if and only if $\operatorname{codim}J\geq 1$; see \cite[Theorem 18.15]{Eisenbud}.
\new{It therefore suffices} to find a tuple $\ell$ of image lines such that $\mathrm{Jac}$ has rank equal to $c$. We prove\footnote{The proof is similar to the computation in \cite[Section 3]{line-multiview}} the existence of such an $\ell$.

Consider a tuple of lines $\ell\in(\mathbb P^2)^m$ such that $M(\ell)$ has rank 2, and such that $\ell\in \Upsilon_{\mathcal C}(\mathbb G)$; i.e., there is a line $L$ mapping to $\ell$, and this line does not pass through any camera center. 
Let $A=(a_{k,\ell})\in\mathbb C^{4\times m}$ be a matrix whose entries are variables that depend on $\ell_{i,j}$, and let $m_1,\ldots,m_N$ be its $3\times 3$ minors. Then, by the chain rule $\mathrm{Jac} = \mathrm J_1\cdot \mathrm J_2$, where 
$$ \mathrm J_1 = \begin{bmatrix}\
    \frac{\partial m_k}{\partial  a_{k,\ell}}\ 
    \end{bmatrix}\in\mathbb C^{N\times (4m)}\quad \text{and}\quad \mathrm J_2 = \begin{bmatrix}\
        \frac{\partial a_{k,\ell}}{\partial  \ell_{i,j}}\ 
        \end{bmatrix}\in\mathbb C^{(4m)\times (3m)}
$$
As $M(\ell)$ has rank 2, the codimension of 
$\mathrm J_1$ 
is equal to the dimension of the variety of rank~2 matrices in $\mathbb C^{4\times m}$, which is $2m-4=c$. Moreover, by linearity
$$\operatorname{Im} \mathrm J_2 = \{[C_1^Tv_1,\ldots, C_m^Tv_m]\in\mathbb C^{4\times m}\mid v_1,\ldots,v_m\in\mathbb C^3\}$$
(here, we have interpreted the image of $ \mathrm J_2$ as a space of matrices).
So, $\operatorname{rank}\mathrm J_2 = 3m$ and we have to show that $\dim \ker \mathrm J_1 \cap\operatorname{Im} \mathrm J_2 = 3m-(2m-4)=m+4$. Notice that $\dim \ker \mathrm J_1 \cap\operatorname{Im} \mathrm J_2 =m+4$ if and only if $\ker \mathrm J_1 $ and $\operatorname{Im} \mathrm J_2$ intersect transversally.

Denote the bilinear form $\langle B_1,B_2\rangle = \mathrm{Trace}(B_1^TB_2)$, and for a subspace $V\subset \mathbb C^{4\times m}$ we denote $V^\perp:=\{B_1\in \mathbb C^{4\times m} \mid \langle B_1,B_2\rangle = 0 \text{ for all } B_2\in V\}$. Then, to show that $\ker \mathrm J_1 $ and $\operatorname{Im} \mathrm J_2$ intersect transversally we can equivalently show that 
$(\ker\mathrm J_1)^\perp \cap (\operatorname{Im} \mathrm J_2)^\perp =0.$ 

We can write 
$\operatorname{Im} \mathrm J_2 = \{P\in\mathbb C^{4\times m} \mid \langle P, c_ie_i^T\rangle = 0,\text{ for } 1\leq i\leq m\},$
where $e_i\in\mathbb R^m$ denotes the $i$-th standard basis vector. 
Assume that $0\neq B=\sum_{k=1}^m\lambda_i c_ie_i^T\in (\ker \mathrm J_1)^\perp$. Without restriction, we assume that $\lambda_1\neq 0$. We show $B\not\in (\ker \mathrm J_2)^\perp$.
The kernel of $\mathrm J_1$ is the tangent space of the variety of rank 2 matrices at $A$. Writing $A = UV^T$ with $U\in\mathbb C^{4\times 2}, V\in\mathbb C^{m\times 2}$ this tangent space is given by all matrices of the form $U\dot V^T + \dot UV^T$ with $\dot U\in\mathbb C^{4\times 2}, \dot V\in\mathbb C^{m\times 2}$. Take $\dot U=0$ and $\dot V = xe_1^T$ with $x=V^TA^*\overline{c_1}$. Then,
$$\langle B, U\dot V^T\rangle=\langle U^TB, \dot V^T\rangle=x^TU^TBe_1 = 
\lambda_1\ (x^TU^Tc_1) = \lambda_1 (\overline{A^Tc_1})^T(A^Tc_1).$$
Recall that $L$ spans the left kernel of $A$. Since $c_1\not\in L$, we have $A^Tc_1\neq 0$, so that $\langle B, U\dot V^T\rangle=\lambda_1 (\overline{A^Tc_1})^T(A^Tc_1)\neq 0$.  This shows $B\not\in (\ker \mathrm J_1)^\perp$. Hence, $\operatorname{rank}\mathrm{Jac}=c$.
\end{proof}

We can now prove \Cref{thm:main_ideal} and \Cref{thm:almost_main_ideal}.
\begin{proof}[Proof of Theorem \ref{thm:main_ideal} and Proposition \ref{thm:almost_main_ideal}]
The zero set of $\IC$ in $(\mathbb C^3)^m$ is $\widehat{\mathcal L_{\mathcal C}}$. Moreover, the ideal~$\IC$ is reduced by \Cref{reduced_ring}. This implies $I(\widehat{\mathcal L_{\mathcal C}})=\IC$, which is the statement of \Cref{thm:almost_main_ideal}.

By the multi-projective Nullstellensatz, $I(\LC)$ is obtained from $\IC=I(\widehat{\mathcal L_{\mathcal C}})$ after saturation with respect to the irrelevant ideal $\bigcup_{i=1}^m V (\ell_i)$. \new{By the first part of \Cref{le: XC},
\[I(\widehat{\mathcal L_{\mathcal C}}) =  I ( X_\mathcal{C}) = I (\widehat{\mathcal L_{\mathcal C}}\setminus \cup_{i=1}^m V (\ell_i)) = I(\widehat{\mathcal L_{\mathcal C}}) : \big(\underbrace{I(\cup_{i=1}^m V (\ell_i))}_\text{irrelevant}\big)^\infty .\] This means that the ideal $\IC$ is already saturated. This proves \Cref{thm:main_ideal}.}
\end{proof}

\medskip
\section{Gröbner Bases for Generic Translational Cameras}\label{sec:groebner}

Let $\mathcal{C}$ be an arrangement of $m$ cameras such that no four cameras are collinear. As before, 
\[\IC := \left\langle 3\times3\text{-minors of } M(\ell)\right\rangle.\]
By \Cref{thm:main_ideal}, proven in the last section, $\IC$ is the ideal of the line multiview variety $\LC$. The purpose of this section is to provide a Gröbner basis for $\IC$ when $\mathcal{C}$ consists of sufficiently generic \emph{translational} cameras.
As we discuss in detail in~\Cref{sec:group}, a generic camera is equivalent to a translational camera up to coordinate change.
This section is \new{also} intended as a warm-up to~\Cref{subsec:collinear-it-ideal}, where similar techniques are used to prove a version of~\Cref{thm:main_ideal}.
Our approach is inspired by the arguments in \cite{atlas}. In this article, the authors work with a certain \emph{symbolic multiview ideal}, where the camera \new{entries} are also variables, and then invoke a specialization argument (see~\cite[Theorem 3.2 and \S 4]{atlas}). We use a similar strategy to obtain a Gr\"obner basis for $\IC$. We begin by defining an analogue of the symbolic multiview ideal in our setting.

\subsection{A Gr\"obner Basis for Partially-Symbolic Multiview Ideals}\label{subsec:it-ideal}

In this section, we study an analogue of the $3\times 3$ minor ideal $\IC$ that is defined for an arrangement of $m\ge 3$ partially-symbolic cameras of a particular form.
Let 
\[\CC [\bl, \bt ] = \CC [\ell_{1, 1}, \ldots \ell_{3, m}, t_{1, 1} , \ldots , t_{3, m}]\] 
denote a polynomial ring in $6m$ indeterminates.
As before, the $3m$ indeterminates $\ell_{i, j}$ represent homogeneous coordinates on the space of $m$-tuples of lines $(\PP^2)^m.$ Let $I_3$ denote the $3\times 3$ identity matrix. We use the other $3m$ indeterminates $t_{i, j}$ to define 
matrices $C(t_1), \ldots , C(t_m)\in \CC [\bl, \bt]^{3\times 4}$ given by
\begin{equation}\label{eq:translational-cameras}
C(t_i) = \begin{bmatrix}
I_3 & t_i
\end{bmatrix}
=
\begin{bmatrix}
1 & 0 & 0 & t_{1,i } \\
0 & 1 & 0 & t_{2,i } \\
0 & 0 & 1 & t_{3,i } 
\end{bmatrix}.
\end{equation}
By analogy with $\IC$, we define $\ICtm$ to be the $3\times  3$ minor ideal associated with the symbolic arrangement $\mathcal C(\bt) := (C(t_1), \ldots , C(t_m))\in (\CC[\bt]^{3\times 4})^m$:
\begin{equation}\label{eq:ICt}
\ICtm = \left\langle \text{$3\times 3$-minors of }  \begin{bmatrix} C(t_1)^T\ell_{1} & \cdots  & C(t_m)^T\ell_{m} \end{bmatrix} \right\rangle.
\end{equation}
We call $\ICtm$ the \emph{indeterminate translation ideal} or the \emph{IT ideal} for short. The motivation for considering camera matrices of the form (\ref{eq:translational-cameras}) is that we can always choose coordinates on $(\mathbb P^2)^m$ and $\mathbb P^3$ such that the camera matrices have this form. Choosing coordinates corresponds to acting by $\mathrm{PGL}_3^m\times \mathrm{PGL}_4$ on the space of $m$-tuples of camera matrices $(\mathbb C^{3\times 4})^m$ via $(H_1,\ldots,H_m,H)\cdot (C_1,\ldots,C_m) := (H_1C_1H,\ldots, H_mC_mH)$. This action is studied in \Cref{sec:group}.

Recall that the \new{Graded Reverse Lex (GRevLex)} order is defined as follows. Monomials are identified with their exponent vector in $\mathbb N^n$. For $\alpha_1,\alpha_2\in \mathbb{N}^{n}$, we say $\alpha_1 >_{\text{GRevLex}} \alpha_2$ if $|\alpha_1|>|\alpha_2|$
or $|\alpha_1|= |\alpha_2|$ and the rightmost nonzero entry of $\alpha_1-\alpha_2\in\mathbb{Z}^{n}$ is negative. We will describe, for any number of cameras $m\ge 3,$ a Gr\"{o}bner basis for $\ICtm$ with respect to a particular monomial order $<,$ defined to be the product of GRevLex orders on the subrings $\CC [\bl]$ and $\CC [\bt ].$
In other words, the monomial order $<$ is defined as follows:
\begin{equation}\label{eq:product-order-def}
\bl^{\alpha_1} \bt^{\beta_1} < \bl^{\alpha_2} \bt^{\beta_2}\quad  
\text{ if }
(\bl^\alpha_1 <_{\text{GRevLex}} \bl^\alpha_2)
\text{ or }
(\alpha_1 = \alpha_2 
\text{ and }
\bt^{\beta_1} <_{\text{GRevLex}} \bt^{\beta_2}).
\end{equation}
For a $k$-element set $\sigma = \{ \sigma_1, \ldots , \sigma_k \} \subset [m]$ we write $\ICt{\sigma_1}{\sigma_k}$ for the IT ideal associated to the cameras $\mathcal{C} (t_{\sigma_1}, \ldots , t_{\sigma_k}).$ Let ${[m]\choose k}$ denote the set of all subsets of $[m]$ of size $k$.
For $3\le k \le m,$ we observe that
\begin{equation}\label{prop:decompose-it-ideal}
\ICtm = \displaystyle\sum_{\sigma \in \binom{[m]}{k}} \ICt{\sigma_1}{\sigma_k}, 
\end{equation}
since the $3\times 3$ minors generating $\ICtm$ also generate the ideal on the right-hand side. Now, for $\sigma,\pi \in \binom{[m]}{k}$ let $G, G'$ be the reduced Gr\"obner bases\footnote{Recall (see e.g.~\cite[\S 2.7, Theorem 9]{CLO15}) that a polynomial ideal has a unique reduced Gr\"{o}bner basis with respect to any monomial order.
A Gr\"{o}bner basis $G$ is said to be \emph{reduced} if every $g\in G$ has leading coefficient~$1$ and, for distinct $g, g' \in G,$ the leading term $\mathrm{in}_< (g)$ does not divide any term of $g'.$} for $\ICt{\sigma_1}{\sigma_k}$ and~$\ICt{\pi_1}{\pi_k}$, respectively. Substituting variables with respect to the monomial order $<$ we get an isomorphism $\ICt{\sigma_1}{\sigma_l}\to \ICt{\pi_1}{\pi_l}$ that maps $G$ to $G'$. Therefore, it suffices to study $\ICt{1}{k}$. This motivates the following. 

\begin{definition}
We denote by $G_k$ the reduced Gr\"{o}bner basis of $\ICt{1}{k}$ with respect to the monomial order $<$.
\end{definition}

The computer algebra system \texttt{Macaulay2}~\cite{M2} allows us to compute $G_k$ for small values of $k$.
As we will argue, the results of these computations for $G_3, \ldots, G_{10}$ allow us to determine the reduced Gr\"{o}bner basis of $\ICtm$ for any number of cameras $m.$
We invite the reader to explore the important case $m=5$ by running the short script in~\Cref{fig:GB-M2}.

\begin{figure}[ht]
\begin{lstlisting}[language=Macaulay2]
m = 5
R = QQ[l_(1,1)..l_(3,m),t_(1,1)..t_(3,m), MonomialOrder => {3*m,3*m}]
linesP2 = for i from 1 to m list matrix{for j from 1 to 3 list l_(j,i)}
cams = for i from 1 to m list id_(R^3) 
                    | matrix for j from 1 to 3 list {t_(j,i)}
rankDropMatrix = matrix{apply(linesP2, cams, (l,c) -> transpose(l*c))}
ITm = minors(3, rankDropMatrix)
Gm = gb ITm
\end{lstlisting}
\caption{\new{\texttt{Macaulay2} code for computing $G_5$}.}\label{fig:GB-M2}
\end{figure}

\noindent \new{The} code in \Cref{fig:GB-M2} lets us inspect $G_m$ for small $m$. The computation reveals the following interesting pattern.
\begin{lemma}\label{lem_5}
Let $2\leq m\leq 10$. Every element of $G_m$ is supported on at most five cameras. More specifically, the reduced Gr\"obner basis $G_m$ is the union of Gr\"obner basis for all subsets of at most $5$ cameras:
\[
    G_m = \displaystyle\bigcup_{\sigma \in \binom{[m]}{5}} \GCt{\sigma_1}{\sigma_5}.
\]
For each $3\leq d\leq 7$, the number of elements of $G_m$ of degree $d$ are listed below.
\begin{center}
    \begin{tabular}{c|c c c c c}
    $d$ & 3&4&5&6&7\\
    \hline
    $\# \{ g \in G_m \mid \deg (g) = d \}$ 
    & $\binom{m}{3}$
    & $3 \cdot \binom{m}{3} $
    & $\binom{m}{4} $
    & $ \binom{m}{4} $
    & $\binom{m+1}{5} $
    \end{tabular}
    \smallskip
\end{center}
\end{lemma}

We now state the main result of this section.
\begin{theorem}\label{thm:GB}
For any $m$, the reduced Gr\"{o}bner basis $G_m$ is equal to the union over all of its restrictions to subsets of $5$ cameras; more precisely,
\begin{equation}\label{eq:m-to-5}
G_m = \displaystyle\bigcup_{\sigma \in \binom{[m]}{5}} \GCt{\sigma_1}{\sigma_5}.
\end{equation}
\end{theorem}
\begin{proof}

Let us write
$
G_m' := \bigcup_{\sigma \in \binom{[m]}{k}} \GCt{\sigma_1}{\sigma_5}
$. The goal is to show $G_m'=G_m$.

By \Cref{lem_5} we have $G_m = G_m'$ for $2\leq m\leq 10$.
For $m\geq 11,$ we apply a variant of Buchberger's S-pair characterization. Let us briefly recall it.
For a finite subset $B\subset R$ we write $f \to_B 0$, if $f$ has a \emph{standard representation} of the form $f =\sum_{i=1}^s h_ i g_i$ such that $h_1,\ldots,h_s\in B$ and 
$\mathrm{in}_< (f) = \max \{ \mathrm{in}_< (h_1 g_1), \, \ldots , \, \mathrm{in}_< (h_s g_s) \}.$ A set of polynomials $B$ is a Gr\"{o}bner basis if and only if for every $g, g' \in B$ we have that~$S(g, g') \to_B\,  0$, where 
\[
S(g, g') := \frac{\mathrm{lcm} (\mathrm{in}_< (g), \mathrm{in}_<(g'))}{\mathrm{in}_<(g)}\, g
-
\frac{\mathrm{lcm} (\mathrm{in}_< (g), \mathrm{in}_<(g'))}{\mathrm{in}_<(g')} \, g'.
\]
Therefore, to show that $G_m'$ is a Gr\"{o}bner basis for $\ICtm$, it suffices to show $S(g, g') \to_{G_m'}\,  0$ \new{for all $g,g'\in G_m'$}. Let $g, g' \in G_m'$. Then, by the definition of $G_m'$ there exists two subsets $\sigma , \pi \in \binom{[m]}{5}$ such that  we have $g\in \GCt{\sigma_1}{\sigma_5},$ $g' \in \GCt{\pi_1}{\pi_5}$. The union of two subsets of size 5 yields a subset of size at most 10. 
We may therefore write 
$\sigma \cup \pi = \{ \sigma_1 ', \ldots , \sigma_k ' \}$ such that $5 \le k \le 10.$
Again by \Cref{lem_5}, $\GCt{\sigma_1'}{\sigma_k'}$ is a Gr\"{o}bner basis, so we must we have 
\begin{align*}
S(g,g') \to_{\GCt{\sigma_1'}{\sigma_k'}} 0.
\end{align*}
But this already implies $S(g, g') \to_{G_m'} 0$, since $g$ and $g'$ only depend on the variables corresponding to $\sigma_1 ', \ldots , \sigma_k '$. This shows that $G_m'$ is a Gr\"obner basis for $<$.

To see that $G_m'$ is reduced, we may again appeal to the cases $m \le 10.$
For any $g$ and~$g'$ as above, $\mathrm{in}_< (g)$ does not divide any term of $g'$ since $\GCt{\sigma_1'}{\sigma_k'}$ is reduced.
Since reduced Gr\"{o}bner bases are unique, we may conclude that $G_m = G_m'.$
\end{proof}

Next, we state a particular property of the Gr\"obner basis $G_m$ in the next proposition. This will be used in the next subsections to determine a Gr\"obner basis for $\IC$ under the explicit genericity conditions given in~\Cref{def:center-generic}. 

\begin{proposition}\label{prop:GB}
The reduced Gr\"{o}bner basis $G_m$ has the following property:
Suppose that~$f (\bt ) \in \CC [\bt]$ is the coefficient of the leading GRevLex monomial in $\CC [\bl ]$ of some element of~$G_m.$ Then $f$ is of one of the four forms listed below:
\begin{itemize}
    \item[  \new{(i)}] $f(\bt)=1$, or
    \item[ \new{(ii)}] $f(\bt) = t_{1,i} - t_{1,j}$, or
    \item[\new{(iii)}] $f(\bt) = (t_{1,i} - t_{1,j}) (t_{1,k} - t_{1,l})$
    \item[\new{(iv)}] $f(\bt )$ is a $3\times 3$ minor of the $4\times m$ matrix of symbolic camera centers:
    \[\begin{bmatrix}
    t_{1,1} & t_{1,2} & \cdots & t_{1,m } \\
    t_{2,1} & t_{2, 2} & \cdots & t_{2,m} \\
    t_{3,1} & t_{3,2} & \cdots & t_{3,m} \\
    -1 & -1 & \cdots  & -1 
    \end{bmatrix}.
    \]
\end{itemize}
\end{proposition}
\begin{proof}
    The statement for $2\leq m\leq 5$ is verified using again \texttt{Macaulay2}. The case $m\geq 5$ follows using \Cref{thm:GB}: the elements in $G_m$ only depend on variables corresponding to at most 5 cameras.  
\end{proof}

\subsection{Specialization to Generic Translational Cameras}\label{sec:specialization}

In this section, we pass from the IT ideal $\ICtm$ to its specialization $\ICtbarm$ by \new{fixing} scalars $\bs = (s_1, \ldots , s_m) \in (\CC^3)^m$.
More formally, let $\ICtbarm$ denote the extension of the ideal $\ICtm$ through the ring homomorphism 
\begin{equation}\label{def_eval_phi}
\phi_{\bs} : \CC [\bl , \bt ]\to \CC [\bl ]
\end{equation}
defined by $\bl \mapsto \bl$ and $\bt \mapsto \bs$; i.e., $\phi_{\bs}$ replaces the $\bt$-variables of polynomials in $\CC [\bl , \bt ]$ by~$\bs$.
Similar to before, 
\[\mathcal{C}(\bs) := (C(s_1), \ldots , C(s_m))\in(\mathbb C^{3\times 4})^m\]
denotes the translational camera arrangement $\ICtbarm$ its associated $3\times 3$ minor ideal. The goal of this section is to prove that for general $\bs$, the image of $G_m$ under $\phi_{\bs}$ is again a Gr\"obner basis. For this, we need a definition.

\begin{definition}\label{def:center-generic}
For $\bs\in(\mathbb C^3)^m$, we say the camera arrangement $\mathcal{C}(\bs)$ is \emph{center-generic} if
\begin{itemize}
    \item[\new{(i)}] $(s_{i_1, j_1} - s_{i_2 ,j_2}) \ne 0$ 
for all $1\le i_1 < i_2 \le 3$ and $1\le j_1, j_2\le m,$
    \item[\new{(ii)}] All $3\times 3$ minors of the matrix~\eqref{eq:center-matrix-specialized} below are nonzero:
    \begin{equation}\label{eq:center-matrix-specialized}
   \begin{bmatrix}
        s_{1, 1} & s_{1,2 } & \cdots & s_{1,m } \\
        s_{2,1} & s_{2 ,2} & \cdots & s_{2,m} \\
        s_{3,1} & s_{3,2} & \cdots & s_{3,m} \\
        -1 & -1 & \cdots  & -1 
        \end{bmatrix}.
        \end{equation}
\end{itemize}
\end{definition}
The genericity conditions of~\Cref{def:center-generic} ensure that all leading coefficients in~\Cref{prop:GB} specialize to nonzero constants.
Note that Condition (i) implies that all products $(s_{i_1 ,j_1} - s_{i_2, j_2}) \cdot (s_{i_3, j_3} - s_{i_4 ,j_4})$ are nonzero.

We now state the main result of this section. 
\begin{theorem}\label{thm:gb-LMI}
Let $\bs\in (\CC^3)^m$, such that $\mathcal C(\bs)$ is center-generic. Let $G_m$ be the Gr\"obner basis from \Cref{thm:GB}. Then, the specialization $\phi_{\bs}(G_m)$ is a Gr\"obner basis for the line multiview ideal $I(\mathcal L_{\ICtbarm})$ with respect to the GRevLex order.
\end{theorem}
\begin{proof}
It follows from the proof of \cite[\S 4.7 Theorem 2]{CLO15} that, if none of the leading coefficients in $\bt$ that appear in $G_m$ vanish at $\bs$, then $\phi_{\bs}(G_m)$ is the Gr\"{o}bner basis for 
\[\phi_{\bs}(\ICtm) = \left\langle \text{$3\times3$-minors of }  \begin{bmatrix} C(s_1)^T\ell_{1} & \cdots  & C(s_m)^T\ell_{m} \end{bmatrix} \right\rangle = \ICtbarm.\] 
Therefore, \Cref{prop:GB} implies that $\phi_{\bs}(G_m)$ is a Gr\"obner basis for  $ \ICtbarm$. Since $G_m$ is a Gröbner basis with respect to the product of GRevLex orders on the subrings $\CC [\bl]$ and~$\CC [\bt ]$, $\phi_{\bs}(G_m)$ is a Gröbner basis with respect to GRevLex on $\CC [\bl]$. 

Furthermore, the camera center $c_j:=\ker C(s_j)$ is spanned by $(s_{1,j}, s_{2,j}, s_{3,j}, -1)$. Since $\ICtbarm$ is center-generic, $\bs$ satisfies condition \new{(ii)}\ in \Cref{def:center-generic}. This implies that no four camera centers are collinear. \Cref{thm:main_ideal} implies $\ICtbarm = I(\mathcal L_{\ICtbarm})$. 
\end{proof}

\medskip
\section{Group Action by Coordinate Change}\label{sec:group} We want to extend the result from the previous section on center-generic translational cameras to any camera arrangement, not necessarily translational, with no three cameras collinear.
Recall from \cref{def_group_action} the action of $G =  \mathrm{PGL}_4\times \mathrm{PGL}_3^m$ on camera arrangements
\begin{align*}
(H,H_1, \ldots, H_m) \cdot (C_1, \ldots, C_m) = (H_1 C_1 H^{-1}, \ldots, H_m C_m H^{-1}). 
\end{align*}
We show that it preserves line multiview ideals in the appropriate way.

Let $a,b\in \PP^3$. Then, 
\begin{align}\label{eq: crossP}
    (H_iC_ia)\times (H_iC_ib)=(\det H_i)H_i^{-T} (C_ia\times C_ib).
\end{align}
This motivates the introduction of the following ring isomorphism: 
\begin{align*}
    L_h : \CC [\bl ] \to \CC [\bl ], \quad
    \ell_i \mapsto H_i^{-T} \ell_i.
\end{align*}
We identify $L_h$ with its map on the level of varieties sending $\ell_i\in \PP^2$ to $H_i^{-T}\ell_i\in \PP^2$. 

\begin{proposition} 
\label{prop: group action}
Let $h = (H,H_1, \ldots, H_m)\in G$ and let $\mathcal{C} = (C_1, \ldots, C_m)$ be a camera arrangement.
Then
$$
\mathcal{L}_{h\cdot \mathcal{C}} = L_h( \LC),
$$
or equivalently
$$
 I(\mathcal{L}_{h\cdot \mathcal{C}})  = L_h^{-1} (I(\LC)).
$$
\end{proposition}

\begin{proof}
The equivalence of the statements in \Cref{prop: group action} comes from the following general fact in commutative algebra:
Given a morphism of projective varieties $\varphi: X \to Y$ there is a corresponding map of graded coordinate rings $\varphi^{\#}: S(Y) \mapsto S(X)$. The ideal of the image $\varphi(X)$ is the kernel of $\varphi^{\#}$. In our setting $\varphi:\mathcal L_{\mathcal C}\to (\mathbb P^2)^m$ is the action of a group element $h$ on the multiview variety $\mathcal L_{\mathcal C}$ and $\varphi^{\#}$ is the composition of $L_h: \mathbb{C}[\ell] \to \mathbb{C}[\ell]$ and the projection $\mathbb{C}[\ell] \to \mathbb{C}[\ell]/I(\LC)$. This implies that the kernel is $L_h^{-1}(I(\LC))= I( L_h(\LC) )$. 

Therefore, it suffices to show $\mathcal{L}_{h\cdot \mathcal{C}} = L_h( \LC )$ to prove the proposition. The argument follows from the following commutative diagram of vector spaces
\begin{center}
 \begin{tikzpicture}[every node/.style={midway}]
        \matrix[column sep={8em,between origins}, row sep={2em}] at (0,0) {
            \node(A) {$\mathbb{C}^4$}; & \node(B) {$\mathbb{C}^{3}$}; ; \\
            \node(C) {$\mathbb{C}^4$}; & \node(D) {$\mathbb{C}^{3}$}; ; \\
        };
        \draw[->] (A) -- (B) node[above]{$C_i$};
        \draw[->] (C) -- (D) node[below]{$ h\cdot C_i $};
        \draw[->] (A) -- (C) node[left]{$H$} ;
        \draw[->] (B) -- (D) node[right]{$ H_i$} ;
    \end{tikzpicture}.
\end{center}
where $h \cdot C_i = H_iC_iH^{-1}$. This induces a commutative diagram of camera maps
\begin{center}
 \begin{tikzpicture}[every node/.style={midway}]
        \matrix[column sep={8em,between origins}, row sep={2em}] at (0,0) {
            \node(A) {\new{$\mathbb G$}}; & \node(B) {$(\mathbb{P}^{2})^m$}; ; \\
            \node(C) {\new{$\mathbb G$}}; & \node(D) {$(\mathbb{P}^{2})^m$}; ; \\
        };
        \draw[dashed,->] (A) -- (B) node[above]{$\Upsilon_{\mathcal{C} }$} ;
        \draw[dashed,->] (C) -- (D) node[below]{$ \Upsilon_{h \cdot \mathcal{C} }$};
        \draw[->] (A) -- (C) node[left]{$\wedge^2 H$} ;
        \draw[->] (B) -- (D) node[right]{$ L_h  $} ;
    \end{tikzpicture}.
\end{center}
The map $\wedge^2 H$ is an isomorphism that sends a line $L$ spanned by $x,y$, to the line spanned by $Hx, Hy$. Now by commutativity, we can compute the closure of the image in the bottom-right corner of the diagram as
$$
    \mathcal{L}_{h\cdot \mathcal{C}} = \overline{\mathrm{Im}(\Upsilon_{h\cdot \mathcal{C} })} = \overline{\mathrm{Im}(\Upsilon_{h\cdot \mathcal{C} } \circ (\wedge^2 H) )} = \
    \overline{\mathrm{Im}(L_h \circ \Upsilon_{\mathcal{C} } ) }.$$
Now, a polynomial $f\in \CC [\ell ]$ vanishes on $\mathrm{Im}(L_h \circ \Upsilon_{\mathcal{C} })$, if and only if $L_h(f)$ vanishes on $\mathrm{Im}(\Upsilon_{\mathcal{C} })$. This implies 
\phantom\qedhere
\[\pushQED{\qed}\mathcal{L}_{h\cdot \mathcal{C}}=\overline{\mathrm{Im}(L_h \circ \Upsilon_{\mathcal{C} } ) } = L_h\big(\overline{(\mathrm{Im}(\Upsilon_{\mathcal{C} })}\big) = L_h (\LC).\qedhere
\popQED\]
\end{proof}

\begin{proposition}\label{prop:when_is_transport_possible}
$\mathcal{C}$ has the property that no three cameras are collinear if and only if there exists $h\in G$ such that $h\cdot \mathcal{C}$ is center-generic.
\end{proposition}

\begin{proof}
First, observe that condition (ii) in~\Cref{def:center-generic} implies that no three centers are collinear when $\mathcal{C}$ is center-generic.
This gives one direction since camera collinearity is a $G$-invariant property.

For the converse, let $\mathcal{C} = (C_1, \ldots, C_m)$ have the property that no three camera centers are collinear.
Consider first the case $m=4.$
If the camera centers  are noncoplanar, then up to the $G$-action we may assume that they form the standard basis $e_1, e_2, e_3, e_4 \in \PP^3.$
In the noncoplanar case, we may assume the centers are $e_1, e_2, e_3, e_1+e_2+e_3 \in \PP^3.$
Correspondingly, we may assume our camera matrices are
\begin{equation}\label{eq:special-cameras}
C_1 = \left[\begin{smallmatrix}
      1&0&0&0\\
      0&1&0&0\\
      0&0&1&0
      \end{smallmatrix}\right] , \quad 
C_2 = \left[\begin{smallmatrix}
      1&0&0&0\\
      0&1&0&0\\
      0&0&0&1
      \end{smallmatrix}\right] , \quad C_3 = \left[\begin{smallmatrix}
      1&0&0&0\\
      0&0&1&0\\
      0&0&0&1
      \end{smallmatrix}\right]
\end{equation}
and 
$$
      C_4 = \begin{cases}
      \left[\begin{smallmatrix}
      0&1&0&0\\
      0&0&1&0\\
      0&0&0&1
      \end{smallmatrix}\right], & \text{ in the non coplanar case}\\[1em]
  \left[\begin{smallmatrix}
      1&0&-1&0\\
      0&1&-1&0\\
      0&0&0&1
      \end{smallmatrix}\right], & \text{ in the coplanar case.}
      \end{cases}$$
In either case, let $H$ be a $4\times 4$ matrix whose entries are the indetermines of the polynomial ring $R=\CC [h_{1,1} , \ldots , h_{4,4}],$ and define $A_i (H) = C_i H$ for $i=1, \ldots , 4.$
The kernel of each matrix $A_i (H)$ is a free $R$-module of rank $1$ generated by some $c_i (H) \in R^{4\times 1}$.
If we construct the matrix 
\[
\left[
\begin{array}{c|c|c|c}
c_1( H) & c_2(H) & c_3 (H) & c_4 (H)
\end{array}
\right] \in R^{4\times 4},
\]
then we may verify by direct computation that the following polynomials defined in terms of this matrix are nonzero:
\begin{itemize}
\item[(i)] all $2\times 2$ minors and the differences between any two entries in the same column, and 
\item[(ii)] all $3\times 3$ minors and the last entry of each column.
\end{itemize}
These conditions correspond to conditions (i) and (ii) in~\Cref{def:center-generic}: specializing $H$ to a generic invertible matrix, we obtain cameras $G$-equivalent to $\mathcal{C}$ which can be made center-generic after acting on the left by the subgroup $\PGL_3^4 \subset G.$

When $m>5,$ fix a set $S = \{ i_1, i_2, i_3, i_4 \}$  with $1\le i_1 < i_2 < i_3 < i_4 \le m.$ 
For all $H$ inside of a dense Zariski open $U_S \subset \CC^{4\times 4}$, we have by the previous argument that $(C_{i_1} H, C_{i_2} H, C_{i_3} H, C_{i_4} H)$ is equivalent to a center-generic $4$-tuple up to left-multiplication.
Thus, if we take \[
H \in \displaystyle\bigcap_{S \in \binom{[m]}{4}} U_S,
\]
then $(C_1 H, \ldots , C_m H)$ is $G$-equivalent to $m$ center-generic cameras.
\end{proof}


\medskip
\section{Set-Theoretic Equations for Line Multiview Varieties}\label{s: Set-Theory} In the case that four or more cameras are collinear, the rank condition of ~\Cref{thm:main_ideal} is not sufficient to describe the line multiview variety\new{,} even set-theoretically. \new{In \cite[Section~2]{line-multiview}, an example is computed with \new{four collinear cameras}
where the rank condition provides two components. One is the line multiview variety, and the other is 4-dimensional and corresponds to the tuples of back-projected planes that all contain the line spanned by the collinear centers. Using elimination of variables in \texttt{Macaulay2}, the authors find one additional equation in the variables of all four lines, that together with the rank condition cuts out the line multiview variety.} Elimination, however, is computationally demanding. Here we describe ideals that set-theoretically determine any line multiview variety with pairwise distinct centers without using elimination. 

Throughout this section $U\vee V$ denotes the linear space spanned by $U$ and $V$.

\subsection{Quadrics of the Line Multiview Variety}
To give equations for the line multiview variety $\mathcal L_{\mathcal C}$ we first need \new{to characterize} points on $\mathcal L_{\mathcal C}$ in terms of associated quadric surfaces. \new{Firstly}, let $\sigma\subset [m]$ index a subset of collinear cameras. Let $E_\sigma$ denote the \emph{baseline} spanned by the collinear camera centers $c_i$ for $i\in \sigma$, and let
$$
 E_\sigma^* := \text{ any line disjoint from }E_\sigma.
$$
For concreteness, one may take the dual line
$\{x \in\mathbb P^3 \mid x^*\, y=0 \text{ for all } y\in E_\sigma\}$
with respect to the Hermitian inner product on $\mathbb C^4$. 

\begin{remark}
In \cite{line-multiview} the authors always use the dual line for $E_\sigma^*$. Nevertheless, the results in \cite{line-multiview} are unchanged when replacing the dual line with any other line, which does not intersect the baseline $E_\sigma$. The reason why we use this more general definition is that in \Cref{lem: action-det-phi} we consider the action of $\mathrm{PGL}_4$ on $\mathbb P^3$. This action does not preserve \new{Hermitian} duality between lines, but it preserves that two lines do not intersect. 
\end{remark}

For $i\in \sigma$ and a line $\ell_i$ in the image plane $\PP^2$ that is in general position with respect to the camera center $c_i,$ we may construct another line $F_\sigma (\ell_i)$ contained in the back-projected plane $H_i$ which passes through $c_i$ and $E_\sigma^*.$
As in~\cite{line-multiview}, we use the notation
\begin{equation}\label{def_F_sigma}
F_\sigma(\ell_i) := c_i\vee \big(H_i\cap E_\sigma^*\big),
\end{equation}
which indeed defines a line provided that $E_\sigma^* \not\subset H_i$.
Otherwise, it is the plane $c_i\vee E_\sigma^*$. In Figure \ref{fig:F_line} we illustrate this definition with a picture.

The next result is a rephrasing of \cite[Theorem 2.6]{line-multiview}.

\begin{theorem}\label{thm: mainQuad} We have $\ell=(\ell_1,\ldots,\ell_m)\in \LC$ if and only if the following three conditions hold:
\begin{enumerate}
    \item all back-projected planes $H_i$ meet in at least a line,
    \item for every maximal set $\sigma\subseteq[m]$ indexing collinear cameras with $|\sigma|\ge 4,$ there is a quadric surface $Q_\sigma=Q_\sigma(\ell)\subseteq \PP^3$, depending on $\ell$, such that $E_\sigma, E_\sigma^*\subseteq Q_\sigma$ and,
    \item if $F_\sigma(\ell_i)$\new{, for some $i\in \sigma$,} is a line, then $F_\sigma(\ell_i)\subseteq Q_\sigma$.
\end{enumerate}
\end{theorem}
\begin{proof} This follows from the proof of \cite[Theorem 2.6]{line-multiview}. We give a summary here. We assume the reader is familiar with the contents of \cite{line-multiview}. Recall that we assume that all centers are distinct and $m\ge 2$.

$\Rightarrow)$ Since conditions 1--3 describe a Zariski-closed subset of $(\PP^2)^m ,$ it is enough to show that a generic point $(\ell_1,\ldots,\ell_m)$ in the image of $\Upsilon_{\mathcal{C}}$ satisfies these three conditions. 
Thus, we may assume that all $F_\sigma(\ell_i)$ are disjoint lines and that the back-projected planes $H_i$ meet in exactly a line $L$ disjoint from both $E_\sigma$ and $E_\sigma^*$. 
This gives the first condition above and also implies $L\neq F_\sigma (\ell_i).$ The quadric $Q_\sigma$ in the second condition is uniquely determined by the property that it contains the pairwise-disjoint lines $E_\sigma,E_\sigma^*,L$. Finally, each of the lines $F_\sigma(\ell_i)$ is contained in the quadric $Q_\sigma$\new{,} because the intersection $F_\sigma (\ell_i) \cap Q_\sigma $ contains the three distinct points\new{,} where $F_\sigma(\ell_i)$ meets $E_\sigma,E_\sigma^*,$ and $L$.

$\Leftarrow)$ Suppose $\ell = (\ell_1, \ldots , \ell_m)$ satisfies conditions 1--3.
Let $L$ be a line where all back-projected planes meet.
If $L$ meets none of the $m$ camera centers, then $\ell$ lies in the image of $\Upsilon_{\mathcal{C}}.$ If $L$ meets exactly one center, then $\ell\in \LC$ by \cite[Lemma 2.8]{line-multiview}. If $L$ meets exactly two centers, then $\ell\in \LC$ by Case 1 of the proof of \cite[Theorem 2.6]{line-multiview}. If $L$ meets three or more centers, we argue as follows. Let $\sigma$ index all camera centers contained in $L$, so that $L = E_\sigma \subset H_i$ for all $i\in \sigma.$ 
It follows that all $F_\sigma(\ell_i),i\in \sigma,$ are lines, since $E_\sigma^*$ is not contained in $H_i$. If the lines $F_\sigma(\ell_i),i\in \sigma,$ are pairwise-disjoint, then the quadric $Q_\sigma$ is smooth and $\ell\in \LC$ by Case 2 of the proof of \cite[Theorem 2.6]{line-multiview}. 
Otherwise, if $F_\sigma(\ell_i)$ and $F_\sigma (\ell_j)$ meet for distinct $i, j\in \sigma$, then $Q_\sigma$ contains the plane $P_1$ spanned by these two lines. 
It follows that $Q_\sigma$ is the union of two planes $P_1\cup P_2$,  where $P_1$ contains $E_\sigma$ and $P_2$ contains $E_\sigma^*$. 
At most one of the lines $F_\sigma(\ell_i)$ can lie in $P_2$.
Otherwise, two centers of $\sigma$ would lie in $P_2$, which would imply that $P_2$ contains $E_\sigma$, but there is no plane that contains both $E_\sigma$ and $E_\sigma^*$. The fact that $\ell\in \LC$ now follows from arguments of Case 3 of the proof of \cite[Theorem 2.6]{line-multiview}.
\end{proof}

\begin{figure}
\begin{center}
  \begin{tikzpicture}[rotate around x=10, rotate around y=25, rotate around z=0, scale = 1]
  \coordinate (p) at (-1.5,0,-2);
  \coordinate (q) at (0,-1,-3);
  \coordinate (c) at (-4,0,3);
  \coordinate (c2) at (1.5,0,2);
  \coordinate (m) at (-1.125,0,2.5);
  \coordinate (a) at (2,0.5,2);
  \coordinate (b) at (-4,0.5,-2);
  \coordinate (u) at (-1,0.5,0);
  \coordinate (v) at (0.2,0.5,2-1.8*2/3);
  \coordinate (x) at (-3,0,1);

  \draw (m) node[below]{\small $E_\sigma$};
  \draw[fill=teal!10] (c) -- (p) -- (q) -- cycle;
  \draw[fill=violet!10] (c2) -- (p) -- (q) -- cycle;
  \draw[thick] (c2) -- (c);
  \draw[thick] (-1.74,0,-1.4) -- (b) node[left]{\small $E_\sigma^*$};
  \draw[thick] (v) -- (a);
  \draw[thick] (u) -- (-0.4,0,-0.545);
   \draw[dashed] (v) -- (-1.7,0,-1.4);
  \draw[very thick, red] (c) -- (u);
  \draw[very thick, red] (c2) -- (v);
  \draw[fill] (p) circle (0.1);
  \draw[fill] (q) circle (0.1) node[above left] {\small $L$};
  \draw[fill] (c) circle (0.1) node[left] {\small ${c_1}$};
  \draw[fill] (c2) circle (0.1) node[below] {\small ${c_2}$};
  \coordinate (w1) at (-4+1.05*3,1.05*0.5,3-1.05*3);
  \coordinate (w2) at (1.5-1.65*1.3,1.65*0.5,2-1.65*0.6*2);
   \draw[very thick, red] (w1) -- (u);
  \draw[very thick, red] (w2) -- (v);
  \draw[fill] (u) circle (0.1);
  \draw[fill] (v) circle (0.1);
  \draw (x) node[above left] {\textcolor{teal}{\small $H_1$}};
  \draw (q) node[above right] {\textcolor{violet}{\small $H_2$}};
  \draw (-2,0,1) node[right] {\textcolor{red}{\small $F_\sigma(\ell_1)$}};
  \draw (0,-0.1,0) node[below] {\textcolor{red}{\small $F_\sigma(\ell_2)$}};
  \end{tikzpicture}
\end{center}
\caption{\label{fig:F_line} The picture illustrates the definition of $F_\sigma(\ell_i)$ from (\ref{def_F_sigma}). Here, we have $\{1,2\}\subset \sigma$ and we only show $F_\sigma(\ell_1)$ and $F_\sigma(\ell_2)$. The two backprojected planes $H_1$ and $H_2$ depicted as triangles, intersect in a line~$L$. The line $E_\sigma^*$ intersects the green backprojected plane $H_1$ in a point. The line spanned by this point and $c_1$ is the red line $F_\sigma(\ell_1)$. Similarly,~$E_\sigma^*$ intersects the violet backprojected plane $H_2$ in a point that together with $c_2$ spans the red line~$F_\sigma(\ell_2)$. The three distinct lines $L,E_\sigma$ and $E_\sigma^*$ intersect all $F_\sigma(\ell_i),i\in \sigma$.}
\end{figure}

In order to establish the results of the next subsection we need a second lemma, where we \new{determine}
when the quadric from \Cref{thm: mainQuad} is unique.

\begin{lemma} \label{le:Uniqueness} Let $L,L'$ be two disjoint lines in $\PP^3$. Let $c_1,c_2,c_3\in L$ be distinct points and let~$A_1,A_2,A_3$ be three lines such that $c_i\in A_i$ for $1\leq i\leq 3$ and each $A_i$ meets the line $L'$. There is a unique quadric containing $L, L', A_1, A_2$ and $A_3$ if and only if at least two of the $A_i$ are disjoint. 
\end{lemma}

\begin{proof} A quadric $Q$ containing $L$ and $L'$ must be either a union of two planes or smooth. If it is smooth, then the $A_i$ must be disjoint, and three disjoint lines uniquely determine a quadric in $\PP^3$. So assume $Q$ is the union of two planes. In this case, one of these planes $P$ is the join of two coplanar lines among the $A_i$, say $A_1$ and $A_2.$ If $A_3$ is not contained in~$P$, then the second plane in $Q$ is determined uniquely as the join of $A_3$ and $L'$. Finally, if~$A_3$ is contained in $P$, then there are infinitely many possibilities for the quadric by letting the second plane be any plane containing $L'$. 
\end{proof}


To determine whether the choice of quadric $Q_\sigma$ in~\Cref{thm: mainQuad} is unique, we may apply~\Cref{le:Uniqueness} with $(L,L') = (E_\sigma , E_\sigma^*),$ and $A_1, A_2, A_3$ of the form $F_\sigma (\ell_i)$.

\subsection{A Set-Theoretic Description for Line Multiview Varieties} 
The first step towards computing polynomial equations that cut out $\mathcal L_{\mathcal C}$ in the presence of at least~4 collinear cameras is to  rewrite $\mathcal L_{\mathcal C}$ as a particular intersection. For this, we need a new notation. Let $\sigma\subseteq [m]$ be a subset of indices and let $\mathcal{L}_{\mathcal{C}_\sigma}$ be the multiview variety of the arrangement $\mathcal{C}_\sigma=(C_i)_{i\in \sigma}.$ Then let $\pi_\sigma:(\PP^2)^m\to (\PP^2)^{\sigma}$ be 
the projection onto the factors indexed by $\sigma$. We write
$$\mathcal{L}_{\mathcal{C},\sigma}:=\pi_\sigma^{-1}(\mathcal{L}_{\mathcal{C}_\sigma}).$$

\begin{proposition}\label{prop: tuples} Let $\Sigma$ be the set of all 3-tuples of indices in $[m]$ and those 4-tuples that correspond to collinear \new{cameras}. Then, we have  
\begin{equation}\label{eq:LC-intersection}\LC=\bigcap_{\sigma\in \Sigma} \mathcal{L}_{\mathcal{C},\sigma}.\end{equation}
\end{proposition}

\begin{proof} The back-projected planes of $\ell$ meet in at least a line if and only if $M(\ell)$ has rank at most 2. Recall the rank condition ideal
\[
\IC = \left\langle 3\times3\text{-minors of } M(\ell)\right\rangle,
\]
where $M(\ell)=\begin{bmatrix}
    C_1^T \ell_1 & \cdots & C_m^T \ell_m
    \end{bmatrix}.$
Each of the given generators of this ideal depends on exactly three cameras. Thus, we are done if we can show that the conditions on $4$-tuples $\sigma$ in~\eqref{eq:LC-intersection} imply the existence of a quadric $Q_\sigma$ as in \Cref{thm: mainQuad} and vice versa. 

Fix a maximal set of indices of at least four collinear cameras $\Gamma$. The existence of a quadric $Q_\Gamma$ satisfying conditions 2-3 of \Cref{thm: mainQuad} directly implies the existence of $Q_{\sigma}$ satisfying conditions 2-3 for any subset $\sigma\subset \Gamma$, particularly those of cardinality 4. Towards the other direction, let 
$$\ell\in \bigcap_{\sigma\in \Sigma} \mathcal{L}_{\mathcal{C},\sigma},$$ 
and assume that three of $F_\Gamma(\ell_i),i\in \Gamma,$ are lines that are not coplanar. Let $\sigma'$ denote a set of three such indices. Consider $\sigma:=\sigma'\cup \{i\}\subseteq \Gamma$ for some $i\in \Gamma\setminus \sigma'$. There exist quadric surfaces $Q_{\sigma}$ as in \Cref{thm: mainQuad} by assumption. But three such lines determine uniquely a quadric $Q$ by \Cref{le:Uniqueness}, and therefore $Q_{\sigma}$ is independent of $i$, and we have $Q_\Gamma=Q_{\sigma}$. Two cases remain. Firstly, if all $F_\Gamma(\ell_i), i\in \Gamma$\new{,} lie \new{in} a common plane $P$, then the union of $P$ with any plane containing $E_\Gamma^*$ suffices for $Q_\sigma$. Secondly, if exactly two $F_\Gamma(\ell_i),i\in \Gamma,$ are lines and they are not coplanar, denote by $L_1, L_2$ these two lines, and let $P_1=E_\Gamma\vee L_1$ and $P_2=E_\Gamma^*\vee L_2$. We may then take $Q_\sigma$ to be $P_1 \cup P_2.$
\end{proof}

Proposition \ref{prop: tuples} implies that in order to obtain polynomial equations cutting out $\mathcal L_{\mathcal C}$ it is enough to obtain equations for $\mathcal{L}_{\mathcal{C},\sigma}$ for every subset $\sigma$ that consists of either 3 indices or 4 indices that correspond to 4 collinear cameras. If $\vert \sigma\vert =3$, we can use \cite[Theorem 2.5]{line-multiview} to deduce that  $\mathcal{L}_{\mathcal{C},\sigma}$ is cut out by those $3\times 3$ minors of the $4\times 3$ submatrix of $M(\ell)$ whose columns are indexed by $\sigma$. So, it remains to obtain equations for $\mathcal{L}_{\mathcal{C},\sigma}$ when $\sigma$ consists of~4 indices corresponding to 4 collinear cameras.

Without loss of generality, we may assume that $$\sigma=\{1,2,3,4\}.$$ We first give polynomial equations for when the four lines $F_\sigma(\ell_i)$ lie on a quadric $Q_\sigma$ as in \Cref{thm: mainQuad}. We need additional notation. 
Fix three distinct points $f_1,f_2,f_3$ on a chosen line $E_\sigma^*$ that is disjoint from $E_\sigma$, and write $f=(f_1,f_2,f_3).$ We define, for $i\in \sigma$,
\begin{align}\label{eq:hi-and-ei}
h_i &:= h_i (\ell) = C_i^T\ell_i, \nonumber \\
e_i(\ell_i) &:=c_i-(h_i^T f_2)f_1+(h_i^T f_1)f_2.
\end{align}
As long as $F_\sigma(\ell_i)$ is a line, then $e_i(\ell_i)$ is a point on $F_\sigma(\ell_i)$ which does not lie on $E_\sigma$ or $E_\sigma^*$. This is the main property of $e_i (\ell_i)$ that we will later use. 
Recalling the (affine) Veronese map 
$$\nu: \CC^4\to \CC^{10},\quad (x,y,z,w)^T\mapsto (x^2,y^2,z^2,w^2,xy,xz,xw,yz,yw,zw)^T,$$
we define a $10\times 10$ matrix $\Phi_{\mathcal C, f, \sigma}(\ell)\in\mathbb C^{10\times 10}$ by applying $\nu$ column-wise to the 10 points $c_1, c_2,c_3, f_1, f_2,f_3$ and \new{$e_i (\ell_i)$ as}
$$\Phi_{\mathcal C, f, \sigma}(\ell):= \nu \big(\ \begin{bmatrix}
c_{1} &
c_{2} &
c_3 &
f_1 & 
f_2 &
f_3 &
e_1(\ell_1) & 
e_2(\ell_2) &
e_3(\ell_3)&
e_4(\ell_4)\end{bmatrix}\ \big).$$

The next result shows that the line multiview variety for a set of four collinear cameras $\sigma $ is determined by rank conditions on $M(\ell )$ and this $10 \times 10$ matrix.

\begin{theorem} 
\label{thm: collinearextragenerators}
Let $\vert \sigma\vert =4$ such that the cameras with indices in $\sigma$ are collinear. 
As before, let $E_\sigma^*$ be any fixed line disjoint from $E_\sigma $,
and let $f=(f_1,f_2,f_3)$ be three distinct fixed points of~$E_\sigma^*$. Then,
$$\mathcal{L}_{\mathcal{C},\sigma} = \{\ell\in (\mathbb P^2)^m \mid \operatorname{rank} M(\ell) \leq 2 \text{ and } \det \Phi_{\mathcal C, f, \sigma}(\ell) = 0\}.$$
\end{theorem}

\begin{proof} A quadratic form defining $Q_\sigma$ may be written as
\begin{equation}\label{def_q}
q(x,y,z,w)=\theta^T\,\nu(x,y,z,w), 
\end{equation}
for some nonzero vector $\theta \in \CC^{10}$. We recall once again that if three distinct points of a line lie on a quadric, then the whole line lies on that quadric. Therefore, the conditions
\begin{equation}\label{eq:Psi-columns-c-f}q(c_i)=\theta^T\,\nu(c_i)=0 \quad \text{and} \quad \quad q(f_i)=\theta^T\nu(f_i)=0,\ i=1,2,3\end{equation}
hold if and only if $E_\sigma,E_\sigma^*\subseteq Q_\sigma$. This explains the first 6 columns of the matrix $\Phi_{\mathcal C, f, \sigma}(\ell)$\new{, as we aim to apply \Cref{thm: mainQuad}}. 

Next, we observe that any point of $c_i\vee E_\sigma^*$ is of the form $\alpha_i c_i + \beta_i f_1+ \gamma_i f_2$, where $(\alpha_i:\beta_i:\gamma_i)\in \PP^2$. The points $x$ that lie on $F_\sigma(\ell_i)$ are those of $c_i\vee E_\sigma^*$ such that $h_i^Tx=0$ (this follows directly from the definition of $F_\sigma(\ell_i)$ in (\ref{def_F_sigma})). 
We have
$$h_i^T\,(\alpha_i c_i + \beta_i f_1+ \gamma_i f_2)=h_i^T\,( \beta_i f_1+ \gamma_i f_2)=0,$$
which leaves two alternatives: Either $h_i^Tf_1=h_i^Tf_2=0$ or 
$$(\beta_i:\gamma_i)=\new{(-h_i^T f_2: h_i^T f_1)}. $$
In the first case, we have \new{that} $F_\sigma(\ell_i)=c_i\vee E_\sigma^*$ is a plane. 
Otherwise\new{,} $F_\sigma(\ell_i)$ is a line and lies in $Q_\sigma$ if and only if three distinct point of $F_\sigma(\ell_i)$ lie in $Q_\sigma$. Consider the three points $c_i, e_i(\ell_i)$ and $a_i:=E_\sigma^*\cap F_\sigma(\ell_i)$. Under the assumption that $q_\sigma(c_i)=0, q_\sigma(f_i)=0$ for $i=1,2,3$ we have seen above that $E_\sigma,E_\sigma^*\subseteq Q_\sigma$. It follows that $c_i,a_\sigma\in Q_\sigma$. Then $F_\sigma(\ell_i)\subseteq Q_\sigma$ if and only if 
\begin{equation}\label{eq:Psi-columns-ei}q_\sigma(e_i(\ell_i)) = \theta^T\,\nu(e_i(\ell_i)) =0.\end{equation}
This gives the last 4 columns of the matrix $\Phi_{\mathcal C, f, \sigma}(\ell)$.  

In summary, the quadric $Q_\sigma$ defined by~\eqref{def_q} satisfies the conditions of~\Cref{thm: mainQuad} if and only if $\theta $ satisfies equations~\eqref{eq:Psi-columns-c-f} and~\eqref{eq:Psi-columns-ei}.
In other words,
$$\theta^T\Phi_{\mathcal C, f, \sigma}(\ell)=0,$$
which in turn is equivalent to $\det \Phi_{\mathcal C, f, \sigma}(\ell) = 0$.
\end{proof}

\begin{example}
\label{example: computephi} 
Let $v_1,v_2,v_3,v_4$ be distinct complex numbers. Consider the four collinear cameras of the form 
\begin{align*}
    C_i = \begin{bmatrix}
1 & 0 & 0 & v_i \\
0 & 1 & 0 & 0 \\
0 & 0 & 1 & 0 
\end{bmatrix},
\end{align*}
for $i\in \sigma=\{1,2,3,4\}$. The centers are $c_i=[v_i:0:0:-1]$. Notice that we may substitute the three centers $c_1, c_2, c_3$ with $[1:0:0:0], [0:0:0:1]$ and $[1:0:0:1]$ in the computation of $\Phi_{\mathcal C, f, \sigma}(\ell)$\new{, since the corresponding columns are only there to ensure that} the associated quadric $Q_\sigma$ contains the baseline. \new{Choose the line $E_\sigma^*$ to be} spanned by $f_1 = [0:1:0:0]$ and $f_2 = [0:0:1:0]$. Letting $f_3=[0:1:1:0]$, we can then write explicitly

    \begin{align*}
        \Phi_{\mathcal C, f, \sigma}(\ell) = \begin{bmatrix}
        1& 0& 1& 0& 0& 0& v_1^2& v_2^2& v_3^2& v_4^2\\ 
        0& 0& 0& 0& 0& 0& -l_{3,1}v_1& -l_{3,2}v_2& -l_{3,3}v_3& -l_{3,4}v_4\\ 
        0& 0& 0& 0& 0& 0& l_{2,1}v_1& l_{2,2}v_2& l_{2,3}v_3 & l_{2,4}v_4\\ 
      0& 0& 1& 0& 0& 0& -v_1& -v_2& -v_3& -v_4\\ 
      0& 0& 0& 1& 0& 1& l_{3,1}^2& l_{3,2}^2& l_{3,3}^2& l_{3,4}^2\\ 
      0& 0& 0& 0& 0& 1& -l_{2,1}l_{3,1}& -l_{2,2}l_{3,2}& -l_{2,3}l_{3,3}& -l_{2,4}l_{3,4}\\ 
      0& 0& 0& 0& 0& 0& l_{3,1}& l_{3,2}& l_{3,3}& l_{3,4}\\ 
      0& 0& 0& 0& 1& 1& l_{2,1}^2& l_{2,2}^2& l_{2,3}^2& l_{2,4}^2\\ 
      0& 0& 0& 0& 0& 0& -l_{2,1}& -l_{2,2}& -l_{2,3} & -l_{2,4}\\ 
      0& 1& 1& 0& 0& 0& 1& 1& 1& 1
      \end{bmatrix}.
    \end{align*}

The determinant of this matrix can be rewritten as
    \begin{align*}
         \det  
        \begin{bmatrix}
l_{3,1}v_1& l_{3,2}v_2& l_{3,3}v_3& l_{3,4}v_4\\  
l_{2,1}v_1& l_{2,2}v_2& l_{2,3}v_3& l_{2,4}v_4\\  
l_{3,1}& l_{3,2}& l_{3,3}& l_{3,4}\\ 
l_{2,1}& l_{2,2}& l_{2,3}& l_{2,4}
      \end{bmatrix}.
    \end{align*}

    \hfill$\diamond$
\end{example}
\medskip

As an extension of $I_\mathcal{C}$ as defined in the introduction, we define
    \begin{align}\label{eq: Itilde}
    \widetilde{I}_{\mathcal C} := \langle 3\times 3 \textnormal{ minors of } M(\ell)\rangle + \sum_{\sigma \in \mathcal{J}} \langle\ \det \Phi_{\mathcal C, f, \sigma}(\ell)\ \rangle,
\end{align}
\new{where $\mathcal{J}$ is} the collection of index sets $\sigma$ of four collinear cameras and with $f$, depending on $\sigma$, being three distinct points of the line $E_\sigma^*$. \new{Using \Cref{prop: tuples} and \Cref{thm: collinearextragenerators}, we establish the following result.}
\begin{corollary} \label{cor: set-theoretic-collinear-ideal} Set-theoretically, the line multiview variety is cut out by the ideal $\widetilde{I}_\mathcal{C}$.
\end{corollary}


In \cite{line-multiview}, after the statement of Theorem 2.6, which this section is based on, an example is given of the set-theoretic constraints for a set of four collinear cameras that have been found through elimination. Here we expand on this example by adding an additional camera matrix.  
\begin{example} Consider the collinear cameras
$$C_1=\left[\begin{smallmatrix}1 & 0 & 0 & 0\\0 & 1 & 0 & 0\\0 & 0 & 1 & 0\end{smallmatrix}\right],~
C_2=\left[\begin{smallmatrix}0 & 1 & 0 & 0\\0 & 0 & 1 & 0\\0 & 0 & 0 & 1\end{smallmatrix}\right],~
C_3=\left[\begin{smallmatrix}1 & 0 & 0 & -1\\0 & 1 & 0 & 0\\0 & 0 & 1 & 0\end{smallmatrix}\right],~
C_4=\left[\begin{smallmatrix}1&  0 & 0 & 1\\0 & 1 & 0 & 0\\0 & 0 & 1 & 0\end{smallmatrix}\right],~C_5=\left[\begin{smallmatrix}2 & 0 & 0 & 1\\0 & 1 & 0 & 0\\0 & 0 & 1 & 0\end{smallmatrix}\right]. $$
The centers $c_i$ of $C_i$ lie on the baseline $E_\tau$, $\tau=\{1,2,3,4,5\}$, spanned by $c_1=[0:0:0:1]$ and $c_2=[1:0:0:0]$. To make the equations below easier to read, we write $x=\ell_1,y=\ell_2,z=\ell_3,w=\ell_4$ and $\ell_5=u$. There are ${5\choose 4}=5$ subsets $\sigma\subset \tau$ with 4 elements giving us 5 constraints beyond the rank condition of $M(\ell)$. These 5 constraints are as follows:
\begin{align*}
    0&=2x_3y_2z_2w_2-x_3y_1z_3w_2-x_2y_2z_3w_2 -x_3y_1z_2w_3-x_2y_2z_2w_3+2x_2y_1z_3w_3,\\
    0&=3x_3y_2z_2u_2+2x_3y_1z_3u_2+x_2y_2z_3u_2 +x_3y_1z_2u_3+2x_2y_2z_2u_3-3x_2y_1z_3w_3,\\
    0&=-x_3y_2w_2u_2+2x_3y_1w_3u_2-x_2y_2w_3u_2 -x_3y_1w_2u_3+2x_2y_2w_2u_3-x_2y_1w_3u_3,\\
    0&=-x_3z_3w_2u_2+3x_3z_2w_3u_2-2x_2z_3w_3u_2 -2x_3z_2w_2u_3+3x_2z_3w_2u_3-x_2z_2w_3u_3,\\
    0&=y_2z_3w_2u_2+3y_2z_2w_3u_2-4y_1z_3w_3u_2 -4y_2z_2w_2u_3+3y_1z_3w_2u_3+y_1z_2w_3u_3.
\end{align*}
The ideal $I_\mathcal{C}$, in this case, is not prime. Computing a primary decomposition of this ideal in \texttt{Macaulay2}~\cite{M2}, there is an associated prime generated by the five polynomials above and the five additional polynomials $\det \Phi_{\mathcal C, f, \sigma}(\ell)$.
This is the vanishing ideal of $\LC .$
\hfill$\diamond$
\end{example}

\subsection{\new{Saturation} with Respect to the Irrelevant Ideal}
Analogous to \Cref{sec:generators}, we define the \textit{cone} over the line multiview variety for an arbitrary camera arrangement as the zero set of $\widetilde{I}_{\mathcal C}$:
\begin{align*}
\widetilde{\mathcal L_{\mathcal{C}}}\ :=&\ \{\ell \in (\mathbb C^3)^m \mid f(\ell)=0\text{ for all } f\in \widetilde{I}_{\mathcal C}\}\\
=&\  \widehat{\mathcal L_{\mathcal{C}}} \cap\Big(\bigcap_{\sigma}\ \{\ell \in (\mathbb C^3)^m \mid \det \Phi_{\mathcal C, f, \sigma}(\ell) = 0\}\Big)\new{,}
\end{align*}
where in the second line\new{,} $\widehat{\mathcal L_{\mathcal{C}}}$ is as in (\ref{def_cone_L_C}) and $\sigma$ runs over all sets of indices corresponding to four collinear cameras. 
The main result of this subsection is \new{that if} $\widetilde{I}_\mathcal{C}$ is radical, then it is also saturated with respect to the irrelevant ideal $\bigcup_{i=1}^m V (\ell_i)$.

\begin{proposition}\label{lem: collinear-saturated} Consider
$$X_{\mathcal{C}}:=\{\ell =(\ell_1,\ldots,\ell_m)\in \widetilde{\mathcal L_{\mathcal C}}\mid \ell_i\neq 0 \text{ for } 1\leq i\leq m\}.$$
The following hold:
\begin{enumerate}
    \item $X_{\mathcal{C}}$ is a Zariski dense subset of $\widetilde{\mathcal L_{\mathcal C}}$, meaning $\overline{X_{\mathcal{C}}}=\widetilde{\mathcal L_{\mathcal C}}$.
    \item If $\widetilde{I}_\mathcal{C}$ is radical, then it is also saturated with respect to the irrelevant ideal $\bigcup_{i=1}^m V (\ell_i)$.
\end{enumerate} 
In particular, if no four cameras are collinear we have $\widehat{\mathcal L_{\mathcal C}}=\widetilde{\mathcal L_{\mathcal C}}$ and $I_{\mathcal C}=\widetilde{I}_{\mathcal C}$, so that the results of this proposition hold verbatim for $\widehat{\mathcal L_{\mathcal C}}$ and $I_{\mathcal C}$ in this case.
\end{proposition}

\begin{proof} \new{The first part is shown analogously to the first part of \Cref{le: XC}. Note that also in the setting of this result, the generators of $\widetilde{I}_{\mathcal C}$ that do not involve $\ell_i$, generate the ideal $\widetilde{I}_{\mathcal{C}'}$, where $\mathcal{C}'$ is the arrangement we get by removing $C_i$ from $\mathcal C$.}

We next deal with \new{saturation, analogously to the proof of \Cref{thm:main_ideal}}. Assuming \new{that $\widetilde{I}_\mathcal{C}$ is radical,} we have $I(\widetilde{\mathcal L_{\mathcal C}})=\widetilde{I}_\mathcal{C}$. We conclude by noting that from the first part of the proof\new{,}
\phantom\qedhere
\[\pushQED{\qed}I(\widetilde{\mathcal L_{\mathcal C}}) =  I ( X_\mathcal{C}) = I (\widetilde{\mathcal L_{\mathcal C}}\setminus \cup_{i=1}^m V (\ell_i)) = I(\widetilde{\mathcal L_{\mathcal C}}) : \big(\underbrace{I(\cup_{i=1}^m V (\ell_i))}_\text{irrelevant}\big)^\infty .\qedhere
\popQED\]
\end{proof}

\subsection{Applying the Group Action} 
In Section \ref{sec:group}\new{,} we have discussed how $G=\mathrm{PGL}_4 \times \mathrm{PGL}_3^m $ acts on the ideal $I_{\mathcal C}$ generated by the $3\times 3$-minors of $M(\ell)$. \new{Here}, we study the action on the additional constraints $\det \Phi_{\mathcal C,\sigma,f}(\ell) = 0$ from the previous section.

For $h = (H, H_1, \ldots, H_m) \in G$\new{,} we extend the group action (\ref{def_group_action}) by setting 
$$h \cdot f = (H f_1,\ Hf_2,\ Hf_3),$$
\new{where $f=(f_1,f_2,f_3)$ is} a triple of distinct points on $E_\sigma^*$, $\sigma\subset [m]$. Recall that the action $L_h$ sends $\ell_i\in \PP^2$ to $H_i^{-T}\ell_i$.

\begin{lemma} 
\label{lem: action-det-phi}  Fix $h=(H, H_1, \ldots , H_m) \in \GL_4 \times \GL_3^m$ representing any element of $ G$, and let $\sigma$ be indices of four collinear cameras.  Then
    $$
    \det\Phi_{h\cdot \mathcal C, h\cdot f, \sigma}(L_h(\ell)) = \det (H)^2 \cdot \det \Phi_{\mathcal C, f, \sigma}(\ell).
    $$
In particular, the vanishing of $\det \Phi_{\mathcal C, f, \sigma}(\ell)$ is unaffected by coordinate changes.
\end{lemma}
    
\begin{proof} First we make the natural identification between $\CC^{10}$ and the set of symmetric $4\times 4$ matrices $\mathrm{Sym}^2(\mathbb{C}^4)$,
$$(a_1,\ldots,a_{10})\cong \left[\begin{smallmatrix}a_1 & a_5 & a_6 & a_7\\
a_5 & a_2 & a_8 & a_9\\
a_6 & a_8 & a_3 & a_{10}\\
a_7 & a_9 & a_{10} & a_4\end{smallmatrix}\right].$$
Then we can identify the Veronese embedding with the map
\begin{align*}
     \nu: \mathbb{C}^4 &\to \mathrm{Sym}^2(\mathbb{C}^4), \\
     p &\mapsto p \otimes p=pp^T.
\end{align*}
For $H \in \mathrm{GL}_4$, define the linear map
\begin{align*}
        (H \otimes H): \mathrm{Sym}^2(\mathbb{C}^4)& \to \mathrm{Sym}^2(\mathbb{C}^4) \\
        A &\mapsto HAH^T.
    \end{align*} 
    It is easy to check that this map is bijective. We may view $(H\otimes H)$ as an invertible linear map $\CC^{10}\to\CC^{10}$, meaning an invertible $10\times 10$ matrix, via the identification above. For any vector $p\in \CC^4$, we can then write in $\CC^{10}$ that $\nu(Hp)=(H\otimes H)\nu (p)$.

Observe that the last four columns of $\Phi_{h\cdot \mathcal C, h\cdot f, \sigma}(L_h(\ell))$ are, after  simplification, 
\begin{align*}
    \nu(Hc_i - (\ell_i^T C_if_2) \cdot Hf_1+(\ell_i^T C_if_1)Hf_2), \quad s=1, \ldots , 4.
    \end{align*}
The inputs of these expressions can be written $He_s(\ell_s)$, where $e_s(\ell_s)$ defined with respect to $ (\mathcal{C}, f, \ell)$. It follows that every column of $\Phi_{h\cdot \mathcal C, h\cdot f, \sigma}(L_h(\ell))$ corresponds to the same column of $\Phi_{\mathcal C, f, \sigma}(\ell)$, except multiplied by the invertible matrix $(H\otimes H)$ from the left. In other words, \phantom\qedhere
\[\pushQED{\qed}\det \Phi_{h\cdot \mathcal C, h\cdot f, \sigma}(L_h(\ell))= \det (H\otimes H) \cdot \det\Phi_{\mathcal C, f, \sigma}(\ell).\qedhere
\popQED\]
\end{proof}



\medskip
\section{Gröbner Bases for Collinear Cameras}\label{s: GrobColl}

This section is aimed at studying the ideal $\widetilde{I}_{\mathcal C}$ introduced in Corollary \ref{cor: set-theoretic-collinear-ideal} that cuts out the multiview variety set-theoretically when all centers are collinear. We do this by providing a Gröbner basis for $\widetilde{I}_{\mathcal C} $ in some special cases and by verifying through the \emph{recognition criterion}~\cite[Proposition 2.3]{atlas} that $\widetilde{I}_{\mathcal C}$ is also the vanishing ideal of the multiview variety. We closely follow the ideas in \Cref{subsec:it-ideal}, which can be summarized in the following steps: 
(1) Define an ideal that describes an arrangement of collinear cameras with indeterminate translations, and compute an explicit Gr\"{o}bner basis for this ideal.
(2) Show that this ideal specializes to the line multiview ideal for ``sufficiently generic'' translational cameras on a fixed line. 
(3) Use the action of the group $\PGL_4 \times \PGL_3^m$ and Proposition \ref{prop: group action} to extend our results to general collinear camera arrangements.

\new{We note that it is possible to do similar work for configurations of centers other than generic and collinear ones. For the sake of brevity, we restrict to these two cases in this paper, and leave a more general treatment for future work.}

\subsection{The Indeterminate Collinear Translation Ideal} \label{subsec:collinear-it-ideal}
Here we study an analogue of the $3\times 3$ minor ideal $\IC$ that is defined for a collection of $m\ge 4$ partially-symbolic cameras in a polynomial ring in $4m$ indeterminates,
$$\CC [\bl, \bv ] = \CC [\ell_{1, 1}, \ldots \ell_{3,m}, v_{1} , \ldots , v_{m}].$$
The $3m$ indeterminates $\ell_{i ,j}$ represent homogeneous coordinates on the space of $m$-tuples of lines~$(\PP^2)^m.$ We use the remaining $m$ indeterminates $v_{i}$ to define the tuple of matrices $\mathcal{C}(\bv) = ( C(v_1), \ldots , C(v_m) ) \in (\CC [\bl, \bv]^{3\times 4})^m$
given by
$$
C(v_i) = \begin{bmatrix}
1 & 0 & 0 & v_{i} \\
0 & 1 & 0 & 0 \\
0 & 0 & 1 & 0 
\end{bmatrix}.
$$
The centers of cameras of this form are $(v_i:0:0:-1)$. Following Theorem \ref{thm: collinearextragenerators} we define the \emph{indeterminate collinear translation ideal}, or \emph{ICT ideal} as
$$
\collICT := 
\left\langle 3\times3\text{-minors of } \begin{bmatrix} C(v_1)^T\ell_{1} & \cdots  & C(v_m)^T\ell_{m} \end{bmatrix}\right\rangle + \sum_{\sigma \in \binom{[m]}{4}} \langle\ \det \Phi_{\mathcal C(\bv), f, \sigma}(\ell)\ \rangle .
$$
where the matrix $\Phi_{\mathcal C, f, \sigma}(\ell)$ is as in Theorem \ref{thm: collinearextragenerators}; see also Example \ref{example: computephi}.

Similar to (\ref{eq:product-order-def}) we define $<$ to be the product of GRevLex orders on $\CC [\bl]$ and $\CC [\bv ]$. Let $\widetilde{G}_{\mathcal{C}(v_{\sigma_1}, \ldots, v_{\sigma_k})}$ denote the reduced Gr\"{o}bner basis of $\widetilde{I}_{\mathcal{C}(v_{\sigma_1}, \ldots, v_{\sigma_k})} $ with respect to $<$, corresponding to the subset of cameras indexed by $\sigma$. 
For brevity, we also write $\widetilde{G}_{m}$ or $\widetilde{G}_{\mathcal{C}(\bv)}$ for the Gröbner basis of $\collICT$. 

The next proposition shows that analogous to Theorem \ref{thm:main2}, we get a reduced Gröbner basis by taking the union of smaller Gröbner bases corresponding to all 4-tuples of cameras.

\begin{proposition}\label{prop:collinear-GB}
For any $m\ge 4$, the reduced Gr\"{o}bner basis $\widetilde{G}_{m}$ is equal to the union over all of its restrictions to subsets of $4$ cameras; more precisely,
\begin{equation}\label{eq:m-to-4}
\widetilde{G}_{m} = \displaystyle\bigcup_{\sigma \in \binom{[m]}{4}} \widetilde{G}_{\mathcal{C}(v_{\sigma_1}, \ldots v_{\sigma_4})}.
\end{equation}
We note the following additional properties of $\widetilde{G}_{m}$:
\begin{itemize}
    \item[\textbf{P1}] No element of $\widetilde{G}_{m}$ is divisible by any of the variables $\ell_{1,1} , \ldots , \ell_{3, m}$.
    \item[\textbf{P2}] The leading terms $\mathrm{in}_<(\widetilde{G}_{m})$ are all squarefree monomials.
    \item[\textbf{P3}] Suppose $f (\bv ) \in \CC [\bv]$ is the coefficient of the leading monomial in $\CC [\bl ]$ of some element of $\widetilde{G}_{m}.$ Then $f$ is of one of the four forms listed below:
    \begin{itemize}
        \item[  i.] $f(\bv)= 1$, or
        \item[ ii.] $f(\bv) = v_{a } - v_{b}$, or
        \item[iii.] $f(\bv) = (v_{a } - v_{b }) (v_{c} - v_{d})$
    \end{itemize}
    \end{itemize}
\end{proposition}
\begin{proof}[Proof of~\Cref{prop:collinear-GB}]
The statement and its proof are analogous to Theorem \ref{thm:GB} and Proposition \ref{prop:GB}, respectively. As before, we take the following two steps:
\begin{enumerate}
    \item Verify computationally with the \texttt{Macaulay2}~\cite{M2} script from Figure \ref{fig:collGB-M2} that $\widetilde{G}_m$ has the desired form and that $\textbf{P1}, \textbf{P2}, \textbf{P3}$ hold for $4 \leq m \leq 8$. 
    \item Deduce that the statement holds for all $m$ because the S-pairs of two elements in~$\widetilde{G}_{m}$ will never involve more than 8 cameras. \phantom\qedhere \qed
\end{enumerate}
\end{proof}

\begin{remark}
We could analogously choose to define $C(v_i)$ by replacing the column $(v_i,0,0)^T$ by $(0,v_i,0)^T$ or $(0,0,v_i)^T$. However, note that in the last case, we would need to adjust the choice of order on $\CC[\bl,\bv]$.
\end{remark}

\begin{figure}[H]
\begin{lstlisting}[language=Macaulay2]
m = 8
R = QQ[l_(1,1)..l_(3,m), v_1..v_m, MonomialOrder=>{3*m,m}]
linesP2 = for j from 1 to m list matrix{for i from 1 to 3 list l_(i,j)};
cams = for i from 1 to m list id_(R^3) | matrix {{v_i}, {0}, {0}};
rankDropMatrix = matrix{transpose \ apply(linesP2, cams, (l, c) -> l*c)}
centerMatrix = fold(gens \ ker \ cams, (a,b) -> a|b)
veronese = p -> (
    q := flatten entries p; 
    matrix {flatten apply(4, i-> apply(4-i,j -> q_i*q_(i+j)))}
    )
Id = id_(R^4)
baseline = {matrix Id_0, matrix Id_3,  matrix (Id_0 + Id_3)}
f = {matrix Id_1, matrix Id_2,  matrix (Id_1 + Id_2) }
e = apply(4, i -> centerMatrix_i - f_0*transpose(f_1)*rankDropMatrix_i
                                 + f_1*transpose(f_0)*rankDropMatrix_i)
Phi = matrix {transpose  \ veronese  \ join(baseline, f , e)}
F = det Phi
newgens = apply(subsets(m, 4), I -> sub(F, flatten 
                apply(4, i -> join({v_(i+1) => v_(I_i+1)}, 
                apply(3, j-> l_(j+1, i+1) => l_(j+1, I_i+1))))))
ICTm = minors(3, rankDropMatrix) + ideal(newgens)
tildeGm = gb(ICTm)

\end{lstlisting}
\caption{Computing $\widetilde{G}_m$ for $m=8$ in \texttt{Macaulay2}.}\label{fig:collGB-M2}
\end{figure}

\subsection{Specialization to Generic Collinear Translational Cameras}

Following the same argument and notation presented in Section \ref{sec:specialization} we transfer what we know about the ICT ideal to a generic arrangement of collinear cameras. For this, we first define for a fixed $\bu\in\mathbb C^{m}$ and analogous to (\ref{def_eval_phi}) the ring homomorphism
$$\phi_{\bu}:\mathbb C[\bl, \bv] \to \mathbb C[\bl]$$
which evaluates the $\bv$ variables at $\bu$.
The first main result is the following theorem.

\begin{theorem}\label{GB_collinear}
Let $\bu\in \CC^m$ be a vector of distinct numbers, or equivalently such that $\mathcal C(\bu)$ has distinct camera centers. Let $\widetilde{G}_{m}$ be the Gr\"obner basis from \Cref{prop:collinear-GB}. Then, the specialization $\phi_{\bu}(\widetilde{G}_{m})$ is a Gr\"obner basis for the ideal $\speccollICT$.
\end{theorem}
\begin{proof}
It follows from \cite[Theorem 2, p.220]{CLO15} that, if none of the leading coefficients in~$\bv$ that appear $\widetilde{G}_{m}$ vanish at $\bu$, then $\phi_{\bu}(\widetilde{G}_{m})$ the Gr\"{o}bner basis for 
$\phi_{\bu}(\collICT) = \speccollICT.$
Because the camera centers are distinct, none of the leading coefficients presented in the statement of \Cref{prop:collinear-GB} vanish.
\end{proof}

\begin{proposition} \label{lem: collinear-radical}
With $\bu\in\mathbb C^m$ as above, the ideal $\speccollICT$ is radical. 
\end{proposition}
\begin{proof}
The property \textbf{P2} in Proposition \ref{prop:collinear-GB} and Theorem \ref{GB_collinear} gives us a Gröbner basis for $\speccollICT$ whose leading terms are squarefree. The result now follows by~\cite[Proposition 2.2]{atlas}.
\end{proof}

Finally, we prove a variant of Theorem \ref{thm:main_ideal} for $m$ collinear cameras.

\begin{theorem}\label{thm:coll-center-LMI}
 Let $m\geq 4,$ and consider a camera arrangement $\mathcal{C}(\bu)$ with $m$ distinct camera centers. Then $ \widetilde{I}_{\mathcal{C}(\bu)}$ is the vanishing ideal of the corresponding multiview variety:
 $$
 \speccollICT = I(\mathcal L_{\mathcal{C}(\bu)}).
 $$  
\end{theorem}
 \begin{proof}
We first consider the case of a collinear translational camera arrangement, $\mathcal{C}(\bu) = (C(u_1), \ldots, C(u_m)),$ that has distinct camera centers. To show that $\speccollICT $ is the vanishing ideal, we use a multiprojective form of the Nullstellensatz (see e.g.~\cite[Proposition 2.3]{atlas}).

In \Cref{s: Set-Theory}, we showed that $\speccollICT $ cuts out the variety set-theoretically. We also showed that $\speccollICT  $ is radical (Proposition \ref{lem: collinear-radical}), and thus saturated with respect to the irrelevant ideal (Proposition \ref{lem: collinear-saturated}). Therefore, $\speccollICT$ is the vanishing ideal of the line multiview variety. 
 \end{proof}

\begin{corollary}
    \label{prop: collinear-non-standard}
Any arrangement of collinear cameras $\mathcal{C}$ has $\widetilde{I}_{\mathcal{C}}$ as its vanishing ideal.
\end{corollary}
\begin{proof}
Up to $G$-equivalence, we prove that collinear cameras have the form 
\begin{align}\label{eq: Gequivform}
C(u_i) = \begin{bmatrix}
1 & 0 & 0 & u_{i} \\
0 & 1 & 0 & 0 \\
0 & 0 & 1 & 0 
\end{bmatrix}.
\end{align}
Combining Theorem \ref{thm:coll-center-LMI}, Proposition \ref{prop: group action} and Lemma \ref{lem: action-det-phi} the result follows. 

To begin with, we use a matrix $H \in \PGL_4$ to transform any collinear camera arrangement into a form where the centers are of the form $(u_i:0:0:-1)$. Then, fix a camera $C_i=[A_i|t_i]$ in the arrangement
, where $A_i\in \CC^{3\times 3}$ and $t_i\in \CC^3$. Since the kernel is of the form $(u_i:0:0:-1)$, $t_i$ is a scaling of the first column of $A_i$. So if $\det(A_i)=0$, then $C_i$ would be at most of rank 2. It follows that $A_i\in \mathrm{GL}_3$ and $A_i^{-1}C_i$ is of the form \eqref{eq: Gequivform}. We are now done by letting $h=(H,A_1^{-1},\ldots,A_m^{-1})$. 
 \end{proof}


\medskip
\bibliographystyle{amsplain}
\bibliography{GBbib.bib}
\end{document}

%% file: startingpackages.tex
\usepackage{tabularx} 
\usepackage{amsmath}
\usepackage{graphicx} 
\usepackage{amsmath}
\usepackage{mathtools} 
\usepackage{amsfonts}
\usepackage{amssymb}
\usepackage{amsthm}
\usepackage{tipa}
\usepackage{marvosym}
\usepackage{mathrsfs}
\usepackage{enumitem}
\usepackage{dsfont}
\usepackage{tikz-cd} 
\usepackage{comment}
\usepackage{changepage}
\usepackage[format=plain, font=footnotesize]{caption}
\usepackage[sort]{cite}
\usepackage{authblk}

\hypersetup{
	colorlinks=true,       
	linkcolor=orange,        
	citecolor=purple,        
	filecolor=magenta,     
	urlcolor=blue
}
\usepackage{blindtext}

\newcommand{\PP}{\mathbb{P}}

\newcommand{\CC}{\mathbb{C}}

\usepackage{cleveref}

\theoremstyle{plain}
\newtheorem{theorem}{Theorem}[section] 
\newtheorem{proposition}[theorem]{Proposition}
\newtheorem{lemma}[theorem]{Lemma}

\newtheorem{corollary}[theorem]{Corollary}

\theoremstyle{definition}
\newtheorem{definition}[theorem]{Definition} 
\newtheorem{example}[theorem]{Example}

\theoremstyle{remark}
\newtheorem*{remark}{Remark}